\documentclass[12pt]{amsart}
\usepackage{geometry}
\usepackage{color}
\usepackage{amssymb}
\usepackage{amsthm}
\usepackage{amsmath}
\usepackage{MnSymbol}
\usepackage{mathrsfs}
\usepackage{bussproofs}
\usepackage{tcolorbox} 
\usepackage[all]{xy}
\usepackage{graphicx}
\usepackage{enumerate}
\usepackage{amsaddr}

\newcommand{\mc}[1]{\mathcal{ #1}}

\newcommand{\algf}{\mathbf{F}}

\newcommand{\ra}{\Rightarrow}

\newcommand{\algp}{\mathbf P}

\newcommand{\alga}{\mathbf A}

\newcommand{\algb}{\mathbf B}
\newcommand{\algc}{\mathbf C}
\newcommand{\algd}{\mathbf D}

\newcommand{\SPO}{\mathcal{SPO}}
\newcommand{\K}{\mathbf{K}}

\newcommand{\C}[1]{\mathrm{C}_{#1}}

\newtheorem{theorem}{Theorem}[section]
\newtheorem{definition}[theorem]{Definition}
\newtheorem{lemma}[theorem]{Lemma}
\newtheorem{proposition}[theorem]{Proposition}
\newtheorem{remark}[theorem]{Remark}
\newtheorem{example}[theorem]{Example}
\newtheorem{corollary}[theorem]{Corollary}
\title{On Contextuality and Unsharp Quantum Logic}
\date{}
\author{Davide~Fazio and Raffaele~Mascella}
\email{dfazio2@unite.it, rmascella@unite.it}
\address{\small{Dipartimento di Scienze della Comunicazione, Universit\`a degli Studi di Teramo, Campus ``Aurelio Saliceti'', Via R. Balzarini, 1, 64100, Teramo (TE), Italy}.}
\begin{document}
\maketitle
\begin{abstract}
In this paper we provide a preliminary investigation of subclasses of bounded posets with antitone involution which are ``pastings'' of their maximal Kleene sub-lattices. Specifically, we introduce super-paraorthomodular lattices, namely paraothomodular lattices whose order determines, and it is fully determined by, the order of their maximal Kleene sub-algebras. It will turn out that the (spectral) paraorthomodular lattice of effects over a separable Hilbert space can be considered as a prominent example of such. Therefore, it arguably provides an algebraic/order theoretical rendering of complementarity phenomena between \emph{unsharp} observables. A number of examples, properties and characterization theorems for structures we deal with will be outlined. For example, we prove a forbidden configuration theorem and we investigate the notion of commutativity for modular pseudo-Kleene lattices, examples of which are (spectral) paraorthomodular lattices of effects over \emph{finite-dimensional} Hilbert spaces. Finally, we show that structures introduced in this paper yield paraconsistent partial referential matrices, the latter being generalizations of J. Czelakowski's partial referential matrices. As a consequence, a link between some classes of posets with antitone involution and algebras of partial ``unsharp'' propositions is established.   
\end{abstract}

\section{Introduction}\label{sec:introduction}
G. Birkhoff and J. von Neumann' seminal paper \cite{BV36} recognized that projection operators over a Hilbert space, which may be regarded as yes-no answers to questions concerning measures of physical magnitudes over a quantum system, form an orthocomplemented bounded lattice that, unlike Boolean algebras, need not be distributive. Structures arising from the Birkhoff-von Neumann's approach, later called \emph{orthomodular lattices} (OMLs) \cite{Husimi}, have become the core of the logico-algebraic foundation of quantum mechanics (QM). This stream of research has provided in-depth analysis of logical novelties put in place by QM (think e.g. to Heisenberg's uncertainty principle, see \cite{Mittelstaedt} for an account) as well as solutions to, and explanations of, theoretical difficulties arising from the probabilistic interpretation of quantum theory. In fact, any (non-Boolean) OML $\alga$ provides a logical rendering of complementarity phenomena between observables, as $\alga$ can be regarded as a ``pasting'' of its maximal Boolean ``contexts''/sub-algebras (called \emph{blocks}) which, in turn, play the role of sub-algebras of propositions on measurements over observables. Furthermore, the order relation induced by lattice operations of $\alga$ determines, and it is fully determined by, the order relation induced by its Boolean sub-algebras. Once the order is interpreted as \emph{entailment}, it becomes clear that OMLs provide a logic which is locally (i.e. in a given context/with respect to a given observable) classical although, in general, it need not be so, once contexts and propositions over non-commuting observables are taken as a whole. In other words, OMLs framework provides a formal, logical account of the fact that ``Quantum Mechanics does not admit global, \emph{non-contextual} value attributions to \emph{all} physical quantities at once'' (cf. \cite{Bush2}).

In spite of its intuitive appeal, some doubts concerning the actual adequacy of OMLs' framework to deal with novelties put in place by QT arose in the history of its development quite soon. For at least two reasons. On the one hand, it was discovered that ``concrete'' orthocomplemented lattices of projection operators over separable Hilbert spaces do not generate the whole variety of OMLs (see e.g. \cite[p. 189]{DallaChiara}). Therefore, orthomodular lattices' theory seemed ``too broad'' for dealing with peculiarities of QT. On the other hand, OMLs are \emph{algebras}, i.e. they are endowed with everywhere defined operations. Thus, once regarded as the algebraic semantics of the logic of quantum experimental propositions, OMLs seemed not cope with the idea that conjunctions of quantum experimental propositions need not be defined (even be meaningful) unless they concern with \emph{commeasurable} observables. Thorough discussions on the topic can be found e.g. in \cite{DallaChiara,Suppes}.\\
In order to cope, at least partially, with the above theoretical difficulties, orthomodular posets (OMPs) were introduced. An OMP is nothing but an orthoposet (see e.g. \cite[p. 13]{DallaChiara} and \cite{FaLePa2}) in which the existence of joins is ensured at least for finite sets of mutually orthogonal elements. Moreover, comparable elements satisfy the orthomodular law (see below). Prominent examples of OMPs arise in Mackey's axiomatic foundation of orthodox QT (see e.g. \cite{MK}) as well as in seemingly far apart contexts (e.g. in the realm of transition systems \cite{Bernardinello}). 

The approach outlined above takes the cue from the idea that projection operators over a Hilbert space can be regarded as quantum experimental propositions concerning measurements of physical quantities. However, these measurements are assumed to fulfill heavy requirements. In fact, they are \emph{ideal measurements of first kind} or L\"uders measurements, i.e. they are \emph{repeatable} (first kind) and test nothing but the property they are supposed to (ideal). See e.g. \cite{DallaChiara} for an account. However, ``these \emph{idealizing} and \emph{simplifying assumptions} of considering only \emph{L\"uders measurements} lead to conflict with both \emph{experimental possibilities} and, even
worse, to \emph{theoretical inconsistencies}'' (\cite[p. 932]{GG05}). As a consequence, a generalized notion of a measurement in quantum mechanics has been introduced (see e.g. \cite{Ludwig}). Such a notion brings with it a generalized concept of measurable properties, called \emph{unsharp properties}, whose formal counterparts are \emph{effects}. An effect over a separable Hilbert space $\mathcal{H}$ is nothing but a bounded self-adjoint linear operator $A$ satisfying the following condition, for any $\rho\in\mathbf{D}(\mathcal{H})$, where $\mathbf{D}(\mathcal{H})$ is the set of \emph{density operators} over $\mathcal{H}$: 
$$0\leq\mathrm{Tr}(\rho A)\leq 1.$$
Upon denoting by $\mathcal{E}(\mathcal{H})$ the set of effects over $\mathcal{H}$, it is well known that  $\mathcal{E}(\mathcal{H})$ carries a quite natural (\emph{canonical}) order $\leq_{c}$, upon defining, for any $A,B\in \mathcal{E}(\mathcal{H})$:\[A\leq_{c} B\text{ iff for any }\rho\in\mathbf{D}(\mathcal{H}),\ \mathrm{Tr}(\rho A)\leq\mathrm{Tr}(\rho B).\]
Moreover, ($\mathcal{E}(\mathcal{H}), \leq_{c})$ has $\mathbb{O}$ resp. $\mathbb{I}$ (the null resp. the identity operator) as its least resp. top element. So, it can be regarded as a bounded poset.\\
Over the past years, unsharp properties have been the core of a large number of investigations both from a mathematical (see e.g. \cite{Kadison}) and philosophical perspective. For example, they are a central concept of the \emph{Unsharp Quantum Reality}'s theory (UQR), first introduced by P. Bush and G. Jaeger in \cite{Bush2}, where $\mathcal{E}(\mathcal{H})$ is considered as comprising ``the totality of the ways in which the system may act on its environment, specifically in measure-like interactions'' \cite[p. 1353]{Bush2}. \\    
Addressing a question posed by P. Bush in \cite{Bush1} concerning the possibility of providing a logical structure for unsharp properties, R. Giuntini and H. Greuling \cite{GG05} (and, in an independent manner, by \cite{FB}) introduced the concept of an \emph{effect algebra}. Effect algebras are partial algebras endowed with a (partial) commutative binary operation $\oplus$, indeed, a generalize orthogonal join, and a total unary operation ${}^{\perp}$, which find a ``concrete'' counterpart in the structure $$\mathbf{E}_{\mathcal{H}}=(\mathcal{E}(\mathcal{H}),\oplus,{}^{\perp},\mathbb{O},\mathbb{I}),$$ where, for any $A,B\in\mathcal{E}(\mathcal{H})$, $A^{\perp}=\mathbb{I}-A$, and $A\oplus B$ is defined, and equals $A+B$, provided that $A\leq_{c}B^{\perp}$.
Although their definition looks quite unimpressive, effect algebras enjoy interesting features which have been the subject of deep investigations. For an extensive account, the reader is referred to Dvurec\v enskij and Pulmannov\'a's monograph \cite{DP00}.\\
Among problems arising in effect algebras' theory, the question if effect algebras are indeed unions of their blocks, under a suitable definition of a block, has been one of the most fascinating challenges addressed by quantum logicians over the past years. However, to the best of our knowledge, the literature on the subject still offers partial solutions only. In \cite{Genca}, J. Gen\v{c}a proved that any \emph{homogeneous} effect algebra $\alga$ is the union of its blocks, where, by a block, it is meant a maximal sub-effect algebra with the \emph{Riesz decomposition property}. Moreover, in analogy with orthomodular lattices, every finite compatible set of a homogeneous effect algebra is contained in a block (cf. \cite[Corollary 3.14]{Genca}). As a consequence, since comparable elements are compatible, any pair of comparable elements is contained in a block. Therefore, the order relation induced by homogeneous effect algebras determines, and it is determined by, the order relation induced by its blocks. These results generalize analogous achievements already obtained for orthoalgebras \cite{Navara1} and lattice-effect algebras (see below and \cite{Riecanova1}). Unfortunately, although interesting from a mathematical perspective, homogeneous effect algebras' theory falls short of capturing salient traits of effects over Hilbert spaces, since it can be seen that, for a Hilbert space $\mathcal{H}$ with $\mathrm{dim}(\mathcal{H}) > 1$, $\mathbf{E}_{\mathcal{H}}$ is not homogeneous.\\
It is important to observe that effect algebras are, in general,  non lattice-ordered. In fact, any lattice-ordered effect algebra (i.e. whose underlying order is a lattice) are homogeneous (see \cite[Proposition 2.2]{Genca}). So, in view of the above discourse, examples of lattice-effect algebras (which are not orthomodular lattices) of a particular physical significance seem hard to be found.

Taking advantage of Olson's celebrated results \cite{Olson} (see also \cite{DeGroote}), R. Giuntini, A. Ledda, and F. Paoli introduced \emph{paraorthomodular lattices}, i.e.   non-distributive generalizations of Kalman's Kleene lattices \cite{Kalman, Cignoli, Cignoli1, Cignoli2} satisfying the paraorthomodular law, as a general framework which lends effects be amenable of smooth lattice theoretical investigations. In fact, it turns out that, for any separable Hilbert space $\mathcal{H}$, $\mathcal{E}(\mathcal{H})$ gives rise to a paraorthomodular lattice by  means of the so-called \emph{spectral ordering} which is, in turn, a generalization of the (lattice) ordering on closed subspaces of $\mathcal{H}$. 

A (bounded) spectral family on a separable Hilbert space $\mc H$ with set $\Pi(\mc H)$ of projection operators is a map $M: \mathbb{R}\rightarrow \Pi(\mc H)$ such that:
\begin{itemize}
\item[a.] For any $\lambda,\mu\in \mathbb{R}$, if $\lambda\leq\mu$, then $M(\lambda)\leq M(\mu)$ (monotonicity); 
\item[b.] For any $\lambda\in\mathbb{R}$, $M(\lambda)=\bigwedge_{\mu > \lambda} M(\mu)$ (right continuity);
\item[c.] There exist $\lambda,\mu\in\mathbb{R}$ ($\lambda \leq\mu$) such that, for any $\eta\in\mathbb{R}$, one has
\[
M(\eta)=\begin{cases} \mathbb{O}, & \mbox{if}\ \eta <\lambda \\ \mathbb{I}, & \mbox{if}\ \eta\geq\mu.
\end{cases}
\]
\end{itemize}
For any bounded self adjoint linear operator $A$ of $\mc H$, there exists a \emph{unique} spectral family $A_{(\cdot)}:\mathbb{R}\to\Pi(\mathcal{H})$ such that $$A=\int_{-\infty}^{\infty}x dA_{x},$$ where the integral is meant in the sense of Riemann-Stieltjes norm-converging sums (see \cite[Chap. 1]{Str01}). Moreover, any spectral family determines a \emph{unique} bounded self-adjoint operator according to the previous equality,. Now, we can introduce the spectral ordering on the set of effects $\mc E(\mc H)$ of $\mc H$ upon defining, for any $E,F\in \mc E(\mc H)$,  
\[
E\leq_{s} F\ \ \mbox{iff}\ \ {F}_{\lambda}\leq {E}_{\lambda},\ \mbox{for}\ \mbox{any}\ \lambda\in\mathbb{R}.\tag{SO}
\]
It can be proven that $\leq_{c}\ \subseteq\ \leq_{s}$. However, $\leq_{s}$ and $\leq_{c}$ do not share some important features. For example, $\leq_{s}$ does not satisfy the following property (see \cite{Olson}): $$E\leq  F\ \mbox{iff}\ F-E\geq  \mathbb{O}.$$ 
It has been shown by \cite{Olson} and \cite{DeGroote} that $\leq_{s}$ turns $\mathcal{E}(\mathcal{H})$ into a (boundedly) complete bounded lattice $(\mathcal{E}(\mathcal{H}),\land_{s},\lor_{s},\mathbb{O},\mathbb{I})$, with bottom resp. top element $\mathbb{O}$ resp. $\mathbb{I}$, upon defining lattice operations over $\mathcal{E}(\mathcal{H})$ componentwise as follows, for any $\lambda\in\mathbb{R}$ (see \cite[p.8]{DeGroote}): 
\begin{enumerate}
\item  $(A\lor_{s}B)_{\lambda}:=A_{\lambda}\land B_{\lambda}$;
\item $(A\land_{s}B)_{\lambda}:=\bigwedge_{\mu>\lambda}(A_{\mu}\lor B_{\mu})$.
\end{enumerate}

Therefore, under the spectral ordering, the whole set of effects is amenable of lattice theoretical analysis. Furthermore, it has been noticed in \cite[p. 13]{DeGroote} that, upon setting, for any $A\in\mathcal{E}(\mathcal{H})$, $A':=\mathbb{I}-A$, one has that $$\mathbf{E}(\mathcal{H})=(\mathcal{E}(\mathcal{H}),\land_{s},\lor_{s},{}',\mathbb{O},\mathbb{I})$$ is a pseudo-Kleene lattice satisfying the \emph{paraorthomodular law} (see below). Moreover, if $\mathcal{H}$ is finite-dimensional, $\mathbf{E}(\mathcal{H})$ is also modular (see Lemma \ref{lem:finitedimmodular}).  Note that, for any $A\in\mathcal{E}(\mathcal{H})$ and $\lambda\in\mathbb{R}$,\[A'_{\lambda}=(\mathbb{I}-A)_{\lambda}= \mathbb{I}-{\bigvee_{\mu< 1-\lambda}B_{\mu}}.\]
The interested reader is referred to \cite[p. 14]{DeGroote} for details. Also, we observe that the complete orthomodular lattice
\[\mathbf{P}(\mathcal{H})=(\Pi(\mathcal{H}),\land,\lor,',\mathbb{O},\mathbb{I})\] of projection  operators over $\mathcal{H}$ is a complete sub-algebra of $\mathbf{E}(\mathcal{H})$ (cf. \cite[Corollary 3.1]{DeGroote}).

In this paper we address the problem of characterizing generalizations of paraorthomodular lattices, here called \emph{unsharp orthogonal posets}, whose induced order determines, and it is determined by, the order of their Kleene-sublattices just as the order of an arbitrary OMP $\alga$ is completely inherited from its Boolean sub-algebras. In this venue, we confine ourselves to pseudo-Kleene lattices, by showing that the class of pseudo-Kleene lattices which are ``pastings'' of their Kleene blocks, here called \emph{super paraorthomodular lattices}, can be axiomatized by means of a rather streamlined equation. We provide several characterizations of the super paraorthomodular law of both order-theoretical and algebraic concern. It will turn out that $\mathbf{E}(\mathcal{H})$ is indeed super-paraorthomodular. As a consequence, $\mathbf{E}(\mathcal{H})$ provides a logico-algebraic account of complementarity phenomena between unsharp observables just as OMLs do in the sharp case. Since the class of OMLs form a sub-variety of super paraorthomodular lattices, we will generalize a number of results which hold for the latter to the wider framework of pseudo-Kleene lattices. For example, we provide a Dedekind-type (forbidden configuration) theorem generalizing well known results for orthomodular posets/lattices, as well as we introduce a notion of commutativity for modular pseudo-Kleene lattices. As it will be clear, unlike plain paraorthomodular lattices, super-paraorthomodular lattices theory captures important features of $\mathbf{E}(\mathcal{H})$. As an instance, it will be shown that sharp elements always form an orthomodular poset.\\
Taking advantage of well known J. Czelakowski's results \cite{Czela1981b}, in the last part of the paper we introduce paraconsistent partial referential matrices for the formal treatment of (partial) experimental quantum propositions taking truth values over the three element Kleene chain, which are \emph{meaningful}/\emph{defined} for a certain state of a physical system, and undefined otherwise. We show that any unsharp orthogonal poset $\alga$ always yield the algebraic reduct of a paraconsistent partial referential matrix. Moreover, if $\alga$ is tame, then it is also \emph{isomorphic} to an algebra of partial propositions. Consequently, due to our characterization results, a novel interpretation of the ``logic'' of super-paraorthomodular lattices, and so of pseudo-Kleene (spectral) lattices of effects over Hilbert spaces, obtains. 

The paper is organized as follows. In Section \ref{sec:preliminaries} we dispatch basic definitions and facts that will be employed in the sequel. Section \ref{sp-orthomodular} introduces super-paraorthomodular lattices. In the same venue, we show that, indeed, super-paraorthomodular lattices coincide with pseudo-Kleene lattices which are ``pastings'' of their maximal Kleene sub-lattices. Also, examples of super-paraorthomodular lattices of particular relevance for logic and the foundation of physics will be exhibited. Section \ref{sec:forbconfandcomm} is devoted to provide a forbidden configuration theorem for super-paraorthomodular lattices and basics of commutativity theory for modular pseudo-Kleene lattices. In Section \ref{sec:paraconspartrefmatr} we introduce paraconsistent partial referential matrices and we provide a rapresentation theorem for tame structures. We conclude in Section \ref{sec:conclusion}. 
\section{Preliminaries and basic facts}\label{sec:preliminaries}
In the sequel, we will assume the reader to be familiar with basics of universal algebra and lattice theory. We refer to \cite{BS} resp. \cite{Gratzer} for details. 

Let $\alga=(A,\land,\lor)$ be a lattice. Moreover, let $\{a,b\}$ be such that $A\cap\{a,b\}=\emptyset$. By $a\oplus\alga\oplus b$ we denote the bounded lattice $(A\cup\{a,b\},\land,\lor)$ such that, for any $x,y\in A\cup\{a,b\}$:\\

\begin{minipage}{3cm}
\begin{equation}
    x\lor y =
    \begin{cases}
      x\lor^{\alga}y & \text{if } x,y\in A \\
      b        &  \text{if }x=b\text{ or }y=b\\
      x        &  \text{if }y=a\\
      y        &  \text{if }x=a
    \end{cases}
\end{equation} 
\end{minipage}\hspace{5cm}
\begin{minipage}{3cm}
\begin{equation}
    x\land y =
    \begin{cases}
      x\land^{\alga}y & \text{if } x,y\in A \\
      a        &  \text{if }x=a\text{ or }y=a\\
      x        &  \text{if }y=b\\
      y        &  \text{if }x=b
    \end{cases}
\end{equation} 
\end{minipage}	\quad\\

The following remark is in order.
\begin{remark}\label{rem:ordsum}Let $\alga$ be a lattice and let $\{a,b\}$ be such that $\{a,b\}\cap A=\emptyset$. Then $\alga$ is distributive if and only if $a\oplus\alga\oplus b$ is.
\end{remark}
Given a (bounded) lattice-ordered algebra $\alga$, we will denote by $\alga^{\ell}$ its (bounded) lattice-reduct.

Let us introduce the main concepts we are deailing with in this paper.
\begin{definition}A \emph{bounded poset with antitone involution} is a structure $\alga=(A,\leq,{}',0,1)$ such that 
\begin{enumerate}
\item $(A,\leq,0,1)$ is a poset with bottom resp. top element $0$ resp. $1$;
\item ${}':A\to A$ is an antitone involution, namely, for any $x,y\in A$, the following hold:
\[x\leq y\text{ implies }y'\leq x',\text{ and }x''=x.\] 
\end{enumerate}
 \end{definition}
\begin{definition}An \emph{unsharp orthogonal poset} (\emph{UOP}, for short) is a structure $\alga=(A,\leq,{}',0,1)$ such that
\begin{enumerate}
\item $(A,\leq,',0,1)$ is a bounded poset with antitone involution;
\item For any $x,y\in A$:
\begin{enumerate}[a.]
\item If $x\leq y'$, then the l.u.b. $x\lor y$ of $x$ and $y$ exists in $\alga$;
\item $x\land x'\leq y\lor y'$ (\emph{Kleene condition}). 
\end{enumerate}
\end{enumerate}
\end{definition}
Given an UOP $\alga$, $a\in A$ is said to be \emph{sharp} provided that $a\land a' = 0$ (and so $a\lor a'=1$). Henceforth, the set of sharp elements of an UOP $\alga$ will be denoted by $\mathrm{Sh}(\alga)$.

Among prominent examples of UOPs already introduced in the literature, fuzzy quantum posets are some of the most studied ones (see e.g. \cite{DvurChov, LeBaLong,LeBaLong1}).

Let $\Omega$ be a set. A \emph{fuzzy set} over $\Omega$ is a map $a:\Omega\to[0,1]$. Let us denote by $F(\Omega)$ the family of fuzzy sets over $\Omega$. As customary, the (countable) union, intersection, and complementation operations over $F(\Omega)$ are defined as follows:
\[(\bigcup_{i=1}^{\infty}a_{i})(x)=\mathrm{sup}\{a_{i}(x):1\leq i<\infty\},\ (\bigcap_{i=1}^{\infty}a_{i})(x)=\mathrm{inf}\{a_{i}(x):1\leq i<\infty\},\ a^{c}(x)=1-a(x).\]
Let $a,b\in F(\Omega)$. We say that 
\begin{itemize}
\item $a$ is \emph{orthogonal} to $b$ ($a\perp b$) provided that $a+b(x)\leq 1$, or, equivalently, $a(x)\leq b^{c}(x)$ (for any $x\in \Omega$);
\item $a$ is \emph{fuzzy orthogonal} to $b$ ($a\perp_{F}b$) provided that $a\cap b(x)\leq\frac{1}{2}$ (for any $x\in \Omega$).
\end{itemize}
 
\begin{definition}Let $\Omega$ be a non empty set and let $M\subseteq [0,1]^{\Omega}$ be such that:
\begin{enumerate}
\item If $\mathbf{1}(x)=1$, for any $x\in\Omega$, then $\mathbf{1}\in M$:
\item If $\mathbf{\frac{1}{2}}(x)=\frac{1}{2}$. for any $x\in\Omega$, then $\mathbf{\frac{1}{2}}\notin M$
\item If $a\in M$, then $a^{c}\in M$.
\end{enumerate}
A couple $(\Omega, M)$ is called a type I (type II) \emph{fuzzy quantum poset} (FQP, for
short), if $M$ is closed with respect to the union of any sequence of fuzzy sets, which are mutually fuzzy orthogonal (orthogonal). 
\end{definition}
\begin{remark}Let $\Omega$ be a non-empty set. Any type II FQP $(\Omega,M)$ is an unsharp orthogonal poset once considered as a structure $(M,\leq,{}^{c},\mathbf{0},\mathbf{1})$ where $\leq$ is fuzzy sets inclusion and $\mathbf{0}=\mathbf{1}^{c}$. Moreover, the same holds for type I FQPs upon noticing that $\perp\subseteq\perp_{F}$.
\end{remark}
  
An UOP $\alga=(A,\leq,{}',0,1)$ such that ${}'$ is a \emph{complementation} is called an \emph{orthogonal poset} \cite{Chajdaorthogonal} (OP, in brief). In other words, an OP $\alga$ is nothing but a UOP such that $\mathrm{Sh}(\alga)=A$. A lattice-ordered UOP (OP) is called a \emph{pseudo-Kleene lattice} (\emph{ortholattice}), see e.g. \cite{Chajda2} (\cite{Beran}). From now on, we will denote the classes of unsharp orthogonal posets, orthogonal posets, pseudo-Kleene lattices and ortholattices by $\mathcal{UOP}$, $\mathcal{OP}$, $\mathcal{PKL}$ and $\mathcal{OL}$, respectively. As customary, any pseudo-Kleene lattice (ortholattice) $\alga$ will be treated as an algebra $(A,\land,\lor,{}',0,1)$ of type $(2,2,1,0,0)$. So, $\mathcal{PKL}$ and $\mathcal{OL}$ are indeed varieties of bounded lattices with antitone involution.\\ 
The variety of \emph{modular} pseudo-Kleene lattices will be denoted by $\mathcal{MPKL}$. It is well-known and worth mentioning that, among prominent examples of modular pseudo-Kleene lattices, we find $\mathbf{E}(\mathcal{H})$, once $\mathcal{H}$ is assumed to be finite-dimensional. For reader's convenience we state and prove this fact in the next lemma.
\begin{lemma}\label{lem:finitedimmodular}
If $\mathcal{H}$ is a finite-dimensional Hilbert space, then $\mathbf{E}(\mathcal{H})$ is a modular pseudo-Kleene lattice. 
\end{lemma}
\begin{proof}
In fact, let $A,B,C\in\mathcal{E}(\mathcal{H})$ be such that $A\leq_{s}B$. Therefore, for any $\lambda\in\mathbb{R}$, $B_{\lambda}\leq A_{\lambda}$. Let $\lambda\in\mathbb{R}$. By the definition of lattice-operations, right monotonicity and,  upon noticing that the lattice $\mathbf{P}(\mathcal{H})$ of projection operators over $\mathcal{H}$ is modular, we compute:
\begin{align*}
[A\lor_{s}(B\land_{s}C)]_{\lambda}=& A_{\lambda}\land (B\land_{s}C)_{\lambda}\\
=& A_{\lambda}\land\bigwedge_{\mu>\lambda}(B_{\mu}\lor C_{\mu})\\
=& \bigwedge_{\mu>\lambda}A_{\mu}\land\bigwedge_{\mu>\lambda}(B_{\mu}\lor C_{\mu})\\
=& \bigwedge_{\mu>\lambda}(A_{\mu}\land(B_{\mu}\lor C_{\mu}))\\
=& \bigwedge_{\mu>\lambda}(B_{\mu}\lor(A_{\mu}\land C_{\mu}))\\
=& \bigwedge_{\mu>\lambda}(B_{\mu}\lor(A\lor_{s}C)_{\mu})\\
=& (B\land_{s}(A\lor_{s}C))_{\lambda}.
\end{align*}
\end{proof}

If a pseudo-Kleene lattice $\alga$ is \emph{distributive}, then it will be called a Kleene lattice. Henceforth, the variety of Kleene lattices will be denoted by $\mathcal{KL}$.\\ It is well known that $\mathcal{KL}$ is generated by the three-element Kleene chain $\K_{3}$ (see e.g. \cite{Balbes}):
\begin{equation}
\xymatrix{
1=0'\ar@{-}[d]\\
\frac{1}{2}=\frac{1}{2}'\ar@{-}[d]\\
0=1'
} \tag{$\K_{3}$}
\end{equation}
\begin{definition}Let $\alga\in\mathcal{UOP}$. $\alga$ is said to be \emph{paraorthomodular} if it satisfies the further condition
\[x\leq y \text{ and } x'\land y=0 \text{ imply } x=y.\] 
\end{definition}
From now on, paraorthomodular UOPs will be simply called  paraorthomodular posets\footnote{Note that our notion of a paraorthomodular poset coincides with the concept of a \emph{sharply paraorthomodular poset} already introduced in the literature, see e.g. \cite{Chajda}.}.
\begin{remark}
Type I or II fuzzy quantum posets are paraorthomodular. For let $(\Omega,M)$ be a (type I or II) fuzzy quantum poset and let $a,b\in M$ be such that $a\leq b$ and $a^{c}\cap b=\mathbf{0}$. One has $b\cap b^{c}=a\cap a^{c}=\mathbf{0}$. Therefore, we have $a,b\in\{0,1\}^{\Omega}$. Let $x\in\Omega$. If $b(x)=0$, then $a(x)=0$, since $a\leq b$. Conversely, if $a(x)=0$, then $0= a^{c}\cap b(x)=\mathrm{inf}\{1-a(x),b(x)\}=b(x)$. By the same reasoning, one has also that $b(x)=1$ iff $a(x)=1$. We conclude $a=b$.
\end{remark}
\begin{example}
Any atomic amalgam of Kleene lattices which does not contain loops of order 3 or 4 forms a paraorthomodular poset \cite{ChFazLeLaPa}.
\end{example}
Orthomodular posets were introduced as abstract counterparts of (partial) algebras of quantum experimental propositions. Such structures generalize the orthocomplemented lattice $\mathbf{P}(\mathcal{H})$ of projection operators over a separable Hilbert space $\mathcal{H}$ upon letting countable joins be defined provided that elements are pairwise orthogonal. Such a requirement ensures a ``logical'' treatment of complementarity between uncompatible observables (see e.g. \cite{Suppes}).

\begin{definition}\emph{(\cite[Theorem 11]{Ka83})}An \emph{orthomodular poset} is a paraorthomodular orthogonal poset.
\end{definition}
Recall that, for orthomodular posets, the paraorthomodularity condition is equivalent to the well-known \emph{orthomodular law}
\[x\leq y\text{ implies }y=x\lor(y\land x').\]
However, this does not hold for paraorthomodular posets in general \cite{GiuLePa}.

A lattice-ordered paraorthomodular (orthomodular) poset will be called a paraorthomodular (orthomodular) lattice. In the sequel,  we denote the class of paraorthomodular (orthomodular) posets and the proper\footnote{See \cite{GiuLePa}.} quasi-variety (variety) of paraorthomodular (orthomodular) lattices by $\mathcal{PMP}$ ($\mathcal{OMP}$) and $\mathcal{POML}$ ($\mathcal{OML}$), respectively. The inclusion relationships between classes of unsharp orthogonal posets encountered so far is summarized by the diagram below. We denote by $\mathcal{BA}$ and $\mathcal{MOL}$ the variety of Boolean algebras and modular ortholattices, respectively.
 
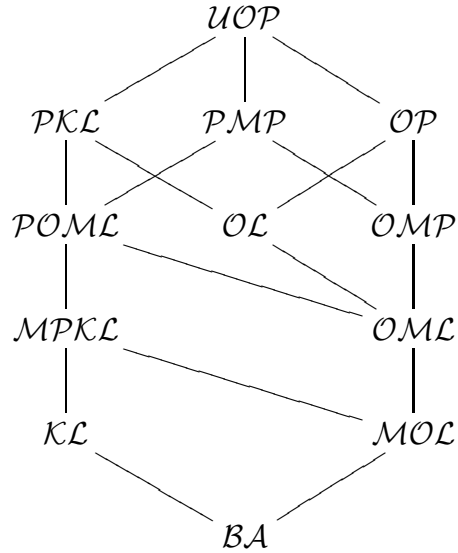
\begin{equation}
\xymatrix{
&\mathcal{UOP}\ar@{-}[dl]\ar@{-}[d]\ar@{-}[dr]&\\
\mathcal{PKL}\ar@{-}[d]\ar@{-}[dr]&\mathcal{PMP}\ar@{-}[dr]\ar@{-}[dl]&\mathcal{OP}\ar@{-}[d]\ar@{-}[dl]\\
\mathcal{POML}\ar@{-}[d]\ar@{-}[drr]&\mathcal{OL}\ar@{-}[dr]&\mathcal{OMP}\ar@{-}[d]\\
\mathcal{MPKL}\ar@{-}[d]\ar@{-}[drr]&&\mathcal{OML}\ar@{-}[d]\\
\mathcal{KL}\ar@{-}[dr]&&\mathcal{MOL}\ar@{-}[dl]\\
&\mathcal{BA}&
} 
\end{equation}

Obviously,  any paraorthomodular lattice is a pseudo-Kleene lattice, but there are pseudo-Kleene lattices which are not paraorthomodular, like e.g. the ortholattice $\algb_{6}$ (the Benzene ring) depicted below:
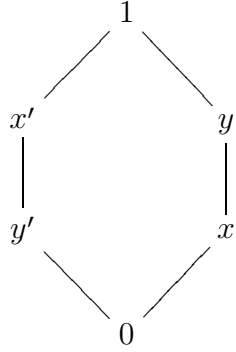
\begin{equation}
\xymatrix{
&1\ar@{-}[dl]\ar@{-}[dr]&\\
x'\ar@{-}[d]&&y\ar@{-}[d]\\
y'\ar@{-}[dr]&&x\ar@{-}[dl]\\
&0&
} 
\end{equation}
The following fact is an easy consequence of \cite[Theorem 2.5]{Chajda}.
\begin{lemma}\label{lem: paraorthom}Let $\alga\in\mathcal{UOP}$. Then $\alga$ is paraorthomodular if and only if it does not contain a subalgebra isomorphic to $\algb_{6}.$
 
\end{lemma}
\begin{definition}
Let $\alga\in\mathcal{UOP}$. A sub-unsharp orthogonal poset $\algb$ of $\alga$ is said to be a Kleene-sublattice of $\alga$ provided that, for any $x,y\in B$, 
\begin{enumerate}
\item $x\land y$ exists in $\alga$ and $x\land y\in B$;
\item $\algb$, regarded as a pseudo-Kleene lattice, is distributive, i.e. it is a Kleene lattice.
\end{enumerate}
\end{definition}
For any $\alga\in\mathcal{UOP}$, we will denote by $\mathfrak{S}(\alga)$ the family of its Kleene sublattices.
\begin{definition}Let $\alga\in\mathcal{UOP}$. A \emph{Kleene block} of $\alga$ is a maximal Kleene sub-lattice of $\alga$.
\end{definition}

Given an unsharp orthogonal poset $\alga$, let us denote by $\mathfrak{B}(\alga)\subseteq \mathfrak{S}(\alga)$ the set of its Kleene blocks. Note that, for any $\alga\in\mathcal{UOP}$, $\mathfrak{B}(\alga)\neq\emptyset$, since $\{0,1\}\subseteq A$ is the universe of a Kleene sub-lattice of $\alga$ which, by Zorn's lemma, is contained in a maximal one. \begin{remark}
Any unsharp orthogonal poset $\alga$ is the union of its Kleene blocks. In fact, for any $x\in A$, the set $\{x,x',x\land x',x\lor x',0,1\}$ is the universe of a Kleene sub-lattice $\algb$ of $\alga$. By Zorn's Lemma, there exists $\algc\in\mathfrak{B}(\alga)$ containing $\algb$. So, any $x\in A$ is contained in a Kleene block. 
\end{remark}
It is not difficult to see that any Kleene sub-lattice of an OP $\alga$ is a Boolean algebra. Therefore, orthogonal posets are union of their Boolean sub-algebras.

\begin{definition}Let $\alga\in\mathcal{UOP}$. $\alga$ is said to be a \emph{pasting} of its blocks provided that, for any $x,y\in A$, $x\leq y$ iff there exists $\K\in\mathfrak{B}(\alga)$ such that $x\leq^{\K}y$.
\end{definition}

In other words, an unsharp orthogonal poset $\alga=(A,\leq,',0,1)$ is the pasting of its blocks if $\leq$ is completely determined by Kleene sublattices of $\alga$. Henceforth, UOPs which are pastings of their Kleene blocks will be called \emph{tame}.

\begin{example}Atomic amalgams of Kleene lattices which do not contain loops of order $3$ or $4$ are tame. This is a consequence of \cite[p. 10]{ChFazLeLaPa}.
\end{example}
The following result provides a characterization of orthomodular posets within the realm of orthogonal posets. Let $\alga=(A,\leq,',0,1)$ and $\algb=(B,\leq,',0,1)$ be bounded posets with antitone involution. Recall that a mapping $f:A\to B$ is an \emph{orthohomomorphism} provided that the following hold, for any $a,b\in A$:
\begin{enumerate}
\item  $a\leq b$ implies $f(a)\leq f(b)$;
\item $f(a')=f(a)'$;
\item $f(1^{\alga})=1^{\algb}$.
\end{enumerate}
$f$ is an \emph{orthoembedding} if it is an orthohomomorphism satisfying the further condition, for any $a,b\in A$, $f(a)\leq f(b)$ implies $a\leq b$, and it is an \emph{orthoisomorphism} if it is a surjective orthoembedding. If $\alga$ and $\algb$ are orthoisomorphic, we write $\alga\cong\algb$.

\begin{lemma}\label{lem: tameorthom}(See e.g. \cite{Beran,Ka83})Let $\alga\in\mathcal{OP}$. The following are equivalent:
\begin{enumerate}
\item $\alga$ is tame;
\item $\alga\in\mathcal{OMP}$;
\item $\alga$ does not contain a sub-orthogonal poset orthoisomorphic to $\algb_{6}$. 
\end{enumerate}
\end{lemma}
As it is well known, the notion of \emph{commutativity} plays an essential role in the theory of orthomodular posets. In fact,  commuting elements of an OMP $\alga$ always generate a Kleene (Boolean) sub-lattice of $\alga$.   

\begin{definition}\emph{(See e.g. \cite[p. 306]{Beran})}Let $\alga\in\mathcal{OMP}$. Then $x,y\in A$ \emph{commute}, written $x\C{}y$, if $x\land y$, $x\land y'$ exist in $\alga$ and $x = (x\land y)\lor(x\land y')$;
\end{definition}

\begin{lemma}\emph{(See e.g. \cite[p. 392]{Pulman})}Let $\alga\in\mathcal{OMP}$ and $x,y\in A$. Then $\{x,y\}\subseteq B$, for some Boolean sub-lattice $\algb$ of $\alga$, if and only if $x\C{}y$.
\end{lemma}

Note that, in general, arbitrary families of mutually commuting elements in an OMP $\alga$ need not be contained in a common Boolean subalgebra (\cite[p. 393]{Pulman}). However, if we assume that $\alga$ is a \emph{quantum logic} (see e.g. \cite{Pulman}), then this holds. Upon recalling that any orthomodular lattice is a quantum logic, we have the following result.
\begin{lemma}Let $\alga\in\mathcal{OML}$. Then, for any $\{a_{i}\}_{i\in I}\subseteq A$, the following are equivalent:
\begin{enumerate}
\item There exists a Boolean subalgebra $\algb$ of $\alga$ such that, $\{a_{i}\}_{i\in I}\subseteq B$;
\item For any $i,j\in I$, $a_{i}\C{}a_{j}$.
\end{enumerate}
 \end{lemma}

An essential tool for practitioners in orthomodular lattices theory is given by the Foulis-Holland theorem.
\begin{theorem}\label{fuli}\emph{(\cite[Proposition 2.8]{BH00})} If $\mathbf{L}$ is an
orthomodular lattice and $a,b,c\in L$ are such that $a$ commutes both with $b$
and with $c$, then the set $\left\{  a,b,c\right\}  $ generates a distributive
sublattice of $\mathbf{L}$.
\end{theorem}
As recalled by Lemma \ref{lem: tameorthom}, an orthogonal poset is tame if and only if it is paraorthomodular, i.e. it is an orthomodular poset. However, once we consider the wider framework of unsharp orthogonal posets, this no longer holds.
\begin{example}
Consider the pseudo-Kleene lattice $\algb_{8}$ depicted below. $\algb_{8}$ is paraorthomodular. However, one has $x\leq y$ but $x$ and $y$ generate a non-distributive (indeed, non-modular) subalgebra of $\algb_{8}$.
\begin{equation}\label{fig:1}
{\tiny\xymatrix{
&1\ar@{-}[d]&\\
&z\ar@{-}[dl]\ar@{-}[dr]&\\
x'\ar@{-}[d]&&y\ar@{-}[d]\\
y'\ar@{-}[dr]&&x\ar@{-}[dl]\\
&z'\ar@{-}[d]&\\
&0&}}\tag{$\algb_{8}$}
\end{equation} 
\end{example}
In the light of the above discussion, it naturally raises the question if a characterization of tame unsharp orthogonal posets generalizing Lemma \ref{lem: tameorthom} can be provided. Indeed, the next section gives a positive answer for pseudo-Kleene lattices. Even more, we will show that the class of pseudo-Kleene lattices which are pastings of their blocks can be axiomatized by means of a rather simple equation which is, indeed, a generalization of the orthomodular law. 
\section{Super paraorthomodular lattices}\label{sp-orthomodular}
In this section we introduce the variety $\mathcal{SPO}$ of super-paraorthomodular lattices which turn out to play for pseudo-Kleene lattices the same role played by orthomodular lattices for ortholattices. In fact, we will show that a pseudo-Kleene lattice is tame if and only if is super-paraorthomodular (Theorem \ref{thm:SPOpastings}). Also, we will provide a number of results clarifying the order-theoretical and algebraic consequences of super-paraorthomodularity (e.g. Theorem \ref{thm: forb config} and Lemma \ref{lem: gigino}). Furthermore, we show that the boundedly complete lattice of effects $\mathbf{E}(\mathcal{H})$ over a separable Hilbert space provides a concrete example of a super paraorthomodular lattice (Theorem \ref{thm: spectrallattice}). As a consequence, algebras we are deailing with might be ranked as genuine quantum structures. Further constructions yielding super-paraorthomodular lattices will be outlined and discussed. Moreover, we will show that sharp elements of structures here introduced always form an orthomodular poset. Indeed, this shows that our framework captures some prominent features of pseudo-Kleene lattices of effects with a certain degree of accuracy. In Section \ref{sec:forbconfandcomm} we will generalize some well known results for orthomodular (modular ortho-)lattices obtained over the past years by providing a forbidden configuration theorem and some remarks on the theory of commutativity.

\begin{definition}A pseudo-Kleene lattice $\alga$ is said to be \emph{super paraorthomodular} (\emph{sp-orthomodular}, for short), provided it satisfies the following conditions:
\begin{enumerate}[{SP}1]
\item $x\leq y$ and $x'\land y=(x\land x')\lor(y\land y')$ imply $y\land(x\lor x')=x\lor(y\land y')$;
\item $x\leq y$ implies $(x\land x')\lor(y\land y')=(x'\land y)\land(x'\land y)'.$ 
\end{enumerate}
\end{definition}
Given a partially ordered set $(P,\leq)$ and $x,y\in P$,  $x\parallel y$ will be short for $x\nleq y$ and $y\nleq x$.
\begin{theorem}\label{thm: forb config}
Let $\alga\in\mathcal{PKL}$. Then $\alga$ satisfies (SP1) iff $\alga$ does not contain a subalgebra isomorphic neither to $\algb_{6}$ nor to $\algb_{8}$.
\end{theorem}
\begin{proof}
As regards the left-to-right direction, if $\alga$ contains a subalgebra isomorphic to $\algb_{6}$, there are $x,y\in A$ such that $x < y$ and $x'\land y=x\land x'=0=y\land y'$. So $x'\land y=(x\land x')\lor(y\land y')$ but $y\land(x\lor x')=y\ne x= x\lor(y\land y')$. If $\alga$ contains a subalgebra  isomorphic to $\algb_8$, we argue similarly, upon noticing that, there are $x,y\in A$ such that $x< y$ and $x'\land y=x\land x'=y\land y'$.\\ Concerning the right-to-left direction, let us assume that there are $x,y\in A$ such that $x\leq y$ and $x'\land y=(x\land x')\lor(y\land y')$ but $y\land (x \lor x')\neq x\lor(y\land y')$. If $x\land x'=0=y\land y'$, then one can argue by means of the same proof employed to show the analogous direction in \cite[Theorem II.5.4]{Beran} (through \cite[Theorem II.3.1]{Beran}), i.e. that $\alga$ contains a subalgebra isomorphic to $\algb_6$. Therefore, we assume w.l.o.g. that $x\land x'\ne 0$ or $y\land y'\ne 0$. We confine ourselves to consider the case $x\land x'\ne 0\ne y\land y'$, since the remaining ones can be treated similarly. One has that $(x'\land x)\lor(y'\land y)=x'\land y=x'\land(y\lor y')\land y\land(x\lor x')\geq (x\lor(y\land y'))\land(y'\lor(x\land x'))\geq (x\land x')\lor(y\land y')$. If $x\leq x'$, then $y\land(x\lor x')=y\land x'=(x\land x')\lor(y\land y')=x\lor (y\land y')$, a contradiction. Analogously, if $y'\leq y$, then $x'\land(y\lor y')=x'\land y=(x\land x')\lor (y\land y')=(x\land x')\lor y'$, again a contradiction. Therefore, we can assume that $x\nleq x'$ and $y'\nleq y$. So, we have $x\lor(y\land y')\parallel y'\lor(x\land x')$, $x\lor(y\land y')\nleq x'\land(y\lor y')$ and $y'\lor(x\land x')\nleq y\land(x\lor x')$, as well as $x\lor(y\land y')\ne(x\land x')\lor(y\land y')\ne y'\lor(x\land x')$. Hence, $\alga$ contains the following subalgebra which is indeed isomorphic to $\algb_{8}$:
\begin{equation}
\xymatrix@C=1pc @R=.8pc {
&1\ar@{-}[d]&\\
&(x\lor x')\land (y\lor y')\ar@{-}[dl]\ar@{-}[dr]&\\
x'\land (y\lor y')\ar@{-}[d]&&y\land (x\lor x')\ar@{-}[d]\\
y'\lor(x\land x')\ar@{-}[dr]&&x\lor(y\land y')\ar@{-}[dl]\\
&(x\land x')\lor(y\land y')\ar@{-}[d]&\\
&0&} 
\end{equation}
\end{proof}
An obvious consequence of Theorem \ref{thm: forb config}, together with Lemma \ref{lem: paraorthom}, is that any sp-orthomodular lattice is paraorthomodular. So, due to characterization results already provided by \cite[Theorem 10.5]{GiuMurPa}, (SP1) turns out to be equivalent to the less cumbersome quasi-equation:
\[x\leq y'\text{ and }x'\land y'\leq x\land y\text{ imply }x=y'.\tag{@}\label{gmp}\]
\begin{lemma}\label{lem: aux1}Any sp-orthomodular lattice satisfies the following condition
\[x\leq y\quad\text{implies}\quad y\land(y'\lor(x\land x'))=(y\land y')\lor(x\land x').\]
\end{lemma}
\begin{proof}
Suppose that  $x\leq y$. $(x\land x')\lor(y\land y')\leq y\land(y'\lor(x\land x'))$ is clear. Moreover, by SP2 one has
\begin{align*}
y\land(y'\lor(x\land x'))=& y\land(y'\lor(x\land x'))\land x'\\
\leq& (y\land x')\land(y'\lor x)\\
=& (x\land x')\lor (y\land y').
\end{align*}
\end{proof}
The next theorem shows that, indeed, (SP1) and (SP2) can be replaced by a single quasi-equation. 
\begin{theorem}\label{eqn: charact}Let $\alga$ be a pseudo-Kleene lattice. Then $\alga$ is sp-orthomodular if and only if the following quasi-equation holds:
\[x\leq y\quad\text{implies}\quad y\land(x\lor x') = x\lor(x'\land y).\label{sp}\tag{sp}\]
\end{theorem}
\begin{proof}Concerning, the left-to-right direction, suppose that $x\leq y$ (and so $y'\leq x'$). Then $x\lor(x'\land y)\leq y\land(x\lor x')$. Now, we have that $(x\lor(x'\land y))'\land(y\land(x\lor x'))=x'\land(x\lor y')\land y\land(x\lor x')=(x'\land y)\land(x\lor y')=(x\land x')\lor(y\land y')$, by SP2. Moreover, note that $(x'\land x)\lor(y\land y')\leq(x'\land(x\lor y'))\land(x\lor(x'\land y))\leq (x\lor(x'\land y))'\land(y\land(x\lor x'))=(x'\land x)\lor(y\land y')$. also, by Lemma \ref{lem: aux1}, one shows that $(y\land(x\lor x'))\land(y'\lor(x\land x'))= y\land(y'\lor(x\land x'))=(y\land y')\lor(x\land x')$. Therefore, one obtains
\begin{align*}
(x\lor(x'\land y))'\land(y\land(x\lor x'))=&(y\land y')\lor(x\land x')\\
=&[(x'\land(x\lor y'))\land(x\lor(x'\land y))]\lor[(y\land(x\lor x'))\land(y'\lor(x\land x'))].
\end{align*}
By an application of SP1, we have:
\begin{align*}
y\land(x\lor x')=& [y\land(x\lor x')]\land[(x\lor x')\land(y\lor y')]\\
=& [y\land(x\lor x')]\land[(x'\land(x\lor y'))\land(x\lor(x'\land y))]'\\
=& (x\lor(x'\land y))\lor[(y\land(x\lor x'))\land(y'\lor(x\land x'))]\\
=& (x\lor(x'\land y))\lor (y\land y')\lor(x\land x')\\
=& x\lor(x'\land y).
\end{align*}
Concerning the right-to left direction, let us first prove SP1. To this aim, suppose that $x\leq y$ and $x'\land y=(x\land x')\lor(y\land y')$. By \eqref{sp}, one has $y\land(x\lor x')=x\lor(x'\land y)=x\lor(x\land x')\lor(y\land y')=x\lor(y\land y')$. As regards SP2, we have that $x\leq y\leq y\lor y'$. So, by \eqref{sp}:
$(y\lor y')\land(x\lor x')=x\lor(x'\land(y\lor y')).$ Moreover, since $y'\leq x'$, again by \eqref{sp}  we have $x'\land (y\lor y')=y'\lor(x'\land y)$. Therefore we have
\[(y\lor y')\land(x\lor x')=(x\lor y')\lor(x'\land y).\] Applying $'$ and De Morgan's laws to both sides of the above equation, the desired result obtains.
\end{proof}
A direct consequence of Theorem \ref{eqn: charact} is that super-paraorthomodularity is equational.
\begin{theorem}\label{rem: sp-equational}A pseudo-Kleene lattice $\alga$ is sp-orthomodular iff the following equation holds: \[(x\lor y)\land(x\lor x')\approx x\lor ((x\lor y)\land x').\]
\end{theorem}
Therefore, sp-orthomodular lattices form a variety that will be denoted by $\mathcal{SPO}$. Note that $\mathcal{OML}\subsetneq\mathcal{SPO}\subsetneq\mathcal{POML}$.

The proof of the following lemma is straightforward. So it is left to the reader.
\begin{lemma}Let $\alga$ be a pseudo-Kleene lattice. Then $\alga$ is sp-orthomodular iff it satisfies the identity \[x\lor((x\lor y)\land (x\lor y)')\approx(x\lor y)\land(x\lor(x\lor y)').\]
\end{lemma}
The next result shows that, indeed, the paraorthomodular  (spectral) lattice $\mathbf{E}(\mathcal{H})$ of effects over a separable Hilbert space $\mathcal{H}$ is a prominent, concrete  example of a sp-orthomodular lattice.
\begin{theorem}\label{thm: spectrallattice}Let $\mathcal{H}$ be a separable Hilbert space. Then $\mathbf{E}(\mathcal{H})=(\mathcal{E}(\mathcal{H}),\land_{s},\lor_{s},{}',\mathbb{O},\mathbb{I})$ is a boundedly complete sp-orthomodular lattice.
\end{theorem}
\begin{proof}It is essentially shown in \cite{DeGroote} (cf. \cite{GiuLePa}) that $\mathbf{E}(\mathcal{H})$ is a boundedly complete pseudo-Kleene lattice. As regards \eqref{sp}, let $A,B\in\mathcal{E}(\mathcal{H})$ with $B\leq_{s}A$. This means that, for any $\lambda\in\mathbb{R}$, one has $A_{\lambda}\leq B_{\lambda}$. Fix $\lambda\in\mathbb{R}$. One has:
\begin{align*}
[B\lor_{s}((\mathbb{I}-B)\land_{s}A)]_{\lambda}=& B_{\lambda}\land [(\mathbb{I}-B)\land_{s}A)]_{\lambda}&\\
=& B_{\lambda}\land\bigwedge_{\mu>\lambda}((\mathbb{I}-B)_{\mu}\lor A_{\mu})&\\
=& B_{\lambda}\land\bigwedge_{\mu>\lambda}[(\mathbb{I}-\bigvee_{\nu<1-\mu}B_{\nu})\lor A_{\mu}]&\\
=& B_{\lambda}\land\bigwedge_{\mu>\lambda}[(\bigwedge_{\nu<1-\mu}\mathbb{I}-B_{\nu})\lor A_{\mu}]&\\
=& \bigwedge_{\mu>\lambda}B_{\mu}\land\bigwedge_{\mu>\lambda}[(\bigwedge_{\nu<1-\mu}(\mathbb{I}-B_{\nu}))\lor A_{\mu}]&\\
=& \bigwedge_{\mu>\lambda}(B_{\mu}\land[(\bigwedge_{\nu<1-\mu}(\mathbb{I}-B_{\nu}))\lor A_{\mu}]).&\\
\end{align*}
Now, let us observe that, due to the monotonicity of a spectral family, projections in its image are comparable. Therefore, they commute, by basic properties of orthomodular lattices. So, one has that $B_{\mu}$ commutes with $B_{\nu}$, for any $\nu<1-\mu$. By elementary properties of commutativity in OMLs (see e.g. \cite[Proposition 2.2]{BH00}), one has that $B_{\mu}$ commutes with $\mathbb{I}-B_{\nu}$, for any $\nu<1-\mu$. By \cite[Theorem 2.14, p. 311]{Beran}, it follows that $B_{\mu}$ commutes with $\bigwedge_{\nu<1-\mu}(\mathbb{I}-B_{\nu})$. Moreover, since $A_{\mu}\leq B_{\mu}$, $B_{\mu}$ commutes with $A_{\mu}$ as well. Therefore, by the Foulis-Holland Theorem, we have $\bigwedge_{\mu>\lambda}(B_{\mu}\land[(\bigwedge_{\nu<1-\mu}(\mathbb{I}-B_{\nu}))\lor A_{\mu}])=\bigwedge_{\mu>\lambda}((B_{\mu}\land(\bigwedge_{\nu<1-\mu}(\mathbb{I}-B_{\nu})))\lor (B_{\mu}\land A_{\mu}))=\bigwedge_{\mu>\lambda}((B_{\mu}\land(\bigwedge_{\nu<1-\mu}(\mathbb{I}-B_{\nu})))\lor A_{\mu})=\bigwedge_{\mu>\lambda}((B_{\mu}\land(\mathbb{I}-B)_{\mu})\lor A_{\mu})=\bigwedge_{\mu>\lambda}[(B\lor_{s}(\mathbb{I}-B))_{\mu}\lor A_{\mu}]=[(B\lor_{s}(\mathbb{I}-B))\land_{s}A]_{\lambda}$. We conclude that, for any $\lambda\in\mathbb{R}$, $[B\lor_{s}((\mathbb{I}-B)\land_{s}A)]_{\lambda}=[(B\lor_{s}(\mathbb{I}-B))\land_{s}A]_{\lambda}$, i.e. $B\lor_{s}((\mathbb{I}-B)\land_{s}A)=(B\lor_{s}(\mathbb{I}-B))\land_{s}A$.
\end{proof}
Examples of sp-orthomodular lattices can be provided by means of a slight modification of  Moisil’s construction of a 3-valued LM algebra (which
is, in particular, a Kleene lattice) from a Boolean algebra, see \cite{Boicescu,Cignoli}. Let $\alga\in\mathcal{OML}$. Consider the set \[A^{[2]}=\{(x,y)\in A^{2}:x\leq y\}.\]
\begin{proposition}Let $\alga=(A,\land,\lor,',0,1)$ be an orthomodular lattice. Then 
$\alga^{[2]} := (A^{[2]},\lor,\land,\sim,(0,0),(1,1))$ is a sp-orthomodular lattice where, for $(a, b), (c, d)\in A^{[2]}$:

\[(a, b) \lor (c, d) := (a \lor c, b \lor d);\ 
(a, b) \land (c, d) := (a \land c, b \land d);\
\sim{(a, b)} := (b', a').\]

\end{proposition}
\begin{proof}Showing that $\alga^{[2]}\in\mathcal{PKL}$ is customary and so it is left to the reader. For example, we show that $\alga^{[2]}$ satisfies regularity. Let $(a,b), (c,d)\in A^{[2]}$. Note that, since $a\leq b$ and $c\leq d$, one has $a\land b'=0$ and $d\lor c'=1$. Therefore, we have $(a,b)\land\sim(a,b)=(a,b)\land(b',a')=(0,b\land a')\leq(c\lor d',1)=(c\lor d', d\lor c')=(c,d)\lor (d',c')=(c,d)\lor\sim(c,d)$. Let us prove \eqref{sp}. Suppose that $(a,b)\leq (c,d)$. This means that $a\leq c\leq d$ as well as $a\leq b\leq d$. We have
\begin{align*}
(c,d)\land((a,b)\lor\sim(a,b)) &= (c,d)\land (a\lor b',b\lor a')\\
&= (c,d)\land (a\lor b',1)\\
&= (c\land(a\lor b'),d)\\
&= ((c\land a)\lor (c\land b'),d)\\
&= (a\lor(c\land b'),(d\lor b)\land (a'\lor b))\\
&= (a\lor(c\land b'),b\lor (a'\land d))\\
&= (a,b)\lor(c\land b',a'\land d)\\
&= (a,b)\lor((c,d)\land(b',a'))\\
&= (a,b)\lor((c,d)\land\sim(a,b)).
\end{align*}
Note that the distributivity applied in the proof follows directly by Theorem \ref{fuli}, upon noticing that $a\mathrm{C}c$, $a\mathrm{C}b'$, $b\mathrm{C}d$  and $b\mathrm{C}a'$. 
\end{proof}
It is well known that, given $\alga\in\mathcal{OML}$, upon setting 
\begin{equation}
 x\odot y:=y\land(x\lor y')\text{ and }x\to y:=x'\lor(x\land y),
\label{defsasak}
\end{equation}
one has that $\odot$ and $\to$ form a \emph{left-residuated} pair, namely, for any $x,y,z\in A$, one has
\begin{equation}
x\odot y\leq z\text{ iff }x\leq y\to z.\label{ressasaki} 
\end{equation}
See \cite{Finch,Chajda1,FaLePa1} for details. Moreover, for any $\alga\in\mathcal{OL}$, upon defining $\odot$ and $\to$ as in \eqref{defsasak}, $\alga$ satisfies \eqref{ressasaki} if and only if $\alga\in\mathcal{OML}$, see e.g. \cite{StJohn}. Now, it naturally raises the question if an analogous characterization holds once $\mathcal{OL}$ is replaced by $\mathcal{PKL}$. And the answer is positive, and very easy to achieve.
\begin{remark}Let $\alga\in\mathcal{PKL}$. Then, upon defining $\odot$ and $\to$ as in \eqref{defsasak}, then $\odot$ and $\to$ satisfy \eqref{ressasaki} if and only if $\alga\in\mathcal{OML}$. In fact, the right-to-left direction follows by \cite{StJohn}. Conversely, one has that $0\leq x\to y$ implies $x\land x' = 0\odot x\leq y$. So $x\land x'=0$, i.e. $\alga\in\mathcal{OL}$. So, our conclusion follows again by \cite{StJohn}.
\end{remark}
However, the same result no longer holds once the following operations are considered:
\begin{equation}\label{defresoperat1}
x\odot y=\begin{cases} 0, & \mbox{if}\ x\leq y' \\ y\land(x\lor y'), & \mbox{otherwise}.
\end{cases}
\end{equation}

\begin{equation}
x\rightarrow y=\begin{cases} 1, & \mbox{if}\ x\leq y \\ x'\lor(x\land y), & \mbox{otherwise}.
\end{cases}\label{defresoperat2}
\end{equation}

In \cite{FaCha}, the problem of providing sufficient (and, possibly, necessary) conditions for a (modular) paraorthomodular lattice $\alga$ to be organized into a \emph{left-residuated $\ell$-groupoid} was considered. 
\begin{definition}
An algebra $\alga=(A,\land,\lor,\odot,\rightarrow,0,1)$ of type $(2,2,2,2,0,0)$ is called a \emph{left-residuated $\ell$-groupoid} if:
\begin{itemize}
\item[(i)] $(A,\land,\lor,0,1)$ is a bounded lattice;
\item[(ii)] $x\odot 1=x=1\odot x$, for any $x\in A$;
\item[(iii)] $x\odot y\leq z$ if and only if $x\leq y\rightarrow z$ (left-residuation), for any $x,y,z\in A$.
\end{itemize}
\end{definition}
In particular, it was shown (\cite[Theorem 3.6]{FaCha}) that any $\alga\in\mathcal{MPKL}$ satisfying the condition \[x\nleq y\text{ and }y\nleq x\text{ imply }x\land x' = y\land y'\] can be turned into a left-residuated $\ell$-groupoid once operations in \eqref{defresoperat1} and \eqref{defresoperat2} are defined. Such a condition becomes also necessary once Kleene lattices, and Boolean-like variants of \eqref{defresoperat1} and \eqref{defresoperat2} over them, are considered (cf. \cite[Theorem 5.3]{FaCha}).\\
However, \cite{FaCha} left open the problem of providing sufficient and necessary conditions for an \emph{arbitrary} pseudo-Kleene lattice to be converted into a left-residuated $\ell$-groupoid upon setting operations as in \eqref{defresoperat1} and \eqref{defresoperat2}. The next result shows that, indeed, the notion of an sp-orthomodular lattice allows to answer the above question in the positive.
\begin{theorem}\label{thm: resleft}Let $\alga\in\mathcal{PKL}$. The following are equivalent:
\begin{enumerate}
\item Upon setting $x\odot y$ and $x\rightarrow y$ as in \eqref{defresoperat1} resp. \eqref{defresoperat2}, the structure $R(\alga)=(A,\land,\lor,\odot,\rightarrow,0,1)$ is a left-residuated $\ell$-groupoid;
\item The following hold:
\begin{enumerate}[(i)]
 \item $\alga\in\SPO$;
 \item For any $x,y\in A$, if $x\nleq y$, then $x\leq y\lor y'$ or $y'\leq x\lor y$.
\end{enumerate}
\end{enumerate}
\end{theorem}
\begin{proof}(1)$\Rightarrow$(2). As regards (i), suppose that $x\leq y$. We show that $y\land (x\lor x')\leq x\lor(y\land x')$. Let us assume w.l.o.g. that $y\land(x\lor x')\nleq x$. One has that $(y\land(x\lor x'))\odot x'=x'\land(x\lor(y\land(x\lor x')))=x'\land(y\land(x\lor x'))=x'\land y\leq y$. By left-residuation, one has $y\land(x\lor x')\leq x'\to y$. Now, if $x'\leq y$, then $y\land(x\lor x')=x\lor x'=x\lor (x'\land y)$. Otherwise, one has that $y\land(x\lor x')\leq x\lor(x'\land y)$. Concerning (ii), suppose that $x\nleq y$. Therefore $x\odot y' = y'\land(y\lor x)\leq y\lor x$. So $x\leq y'\to(y\lor x)$. If $y'\nleq x\lor y$, then $y'\to(y\lor x)=y\lor(y'\land(x\lor y))=(x\lor y)\land(y\lor y')$ (by (i) and \eqref{sp}). We conclude $x\leq y\lor y'$.\quad\\
(2)$\Rightarrow$(1). If $x=0$, then $x\odot 1=x=1\odot x$. If $x\ne 0$, then $x\odot 1=1\land(x\lor 0)=x=x\land(1\lor x')=1\odot x$. Let us prove left-residuation. Let $x\odot y\leq z$. We can assume w.l.o.g. that $y\nleq z$ (otherwise $x\leq y\to z=1$) as well as $x\nleq y'$ (otherwise $x\leq y'\lor(y\land z)$ would be trivially satisfied). We have $y\land(x\lor y')\leq z$ and so, by \eqref{sp}, $(y\lor y')\land(x\lor y')=y'\lor(y\land(y'\lor x))\leq y'\lor(y\land z)$. Now, by (ii), since $x\nleq y'$, two cases are possible. (a) $x\leq y\lor y'$; (b) $y\leq x\lor y'$. If (a), then one has that $x\leq x\lor y'\leq y'\lor (y\land z)=y\to z$. If (b), one would have $y\leq z$, which is impossible by hypothesis. Now, assume that $x\leq y\to z$. If $y\leq z$ or $x\leq y'$, then we are done. So let us consider the case $x\nleq y'$ and $y\nleq z$. We have
\begin{align*}
x\odot y =& y\land(x\lor y')\\
\leq & y\land(y'\lor(y\land z))\\
= & (y\land y')\lor(y\land z).  
\end{align*}
Since $y\nleq z$, one has also $z'\nleq y'$. Two cases are possible: $z'\leq y\lor y'$. So $y\land y'\leq z$, and so $x\odot y\leq y\land z\leq z$. If $y\leq z'\lor y'$, then $z\land y\leq y'$. Therefore, by hypothesis, $x\leq y'$, a contradiction. 
\end{proof}
Furthermore, we observe that the class $\mathcal{RSP}$ of sp-orthomodular lattices satisfying  2(ii) of Theorem \ref{thm: resleft} does not form a quasi-variety.
\begin{example}Consider an isomorphic copy of the sp-orthomodular (indeed, Kleene) lattice $\K_{3}$. A routine check shows that it satisfies condition 2(ii) of Theorem \ref{thm: resleft}. However, upon taking the direct product $\K_{3}\times\algb_{2}$ (where $\algb_{2}$ is the two-element Boolean algebra), this is no longer the case. In fact, if we consider the pairs $a=(1,0),b=(\frac{1}{2},0)$, we have $a\nleq b$ but neither $a\leq b\lor b'=b'=(\frac{1}{2},1)$, nor $b'\leq a\lor b=a$. Therefore, $\mathcal{RSP}$ is not closed under direct products.
\end{example}
The next result shows that, indeed, sp-orthomodular lattices coincide with the class of tame pseudo-Kleene lattices. 
 
\begin{theorem}\label{thm:SPOpastings}Let $\alga\in\mathcal{PKL}$. Then $\alga\in\mathcal{SPO}$ iff $\alga$ is a pasting of $\mathfrak{B}(\alga)=\{\K_{i}\}_{i\in I}$.
\end{theorem}
\begin{proof}
Concerning the right-to-left direction, let $x,y\in A$ be such 
that $x\leq y$. Then, for some $i\in I$, $x\leq^{\K_{i}} y$. 
Therefore $y\land(x\lor x')=(y\land x)\lor(y\land 
x')=x\lor(y\land x')$, since $\K_{i}$ is a distributive 
subalgebra of $\alga$. Let us prove the converse direction. 
Suppose that $a\leq b$. We show that the set ${K}=\{b'\land a,a\land 
a', b\land b', (a\land a')\lor(b\land b'),a,b',a\lor(b\land 
b'),b'\lor(a\land a'),a'\land b,a\lor b',a'\land(b\lor 
b'),b\land(a\lor a'),a',b,(a\lor a')\land(b\lor b'),a\lor 
a',b\lor b', a'\lor b, 0, 1\}$ is the universe of a subalgebra $\K$ of $\alga$ which is distributive. First, by a direct inspection, we observe that $K$ is closed under ${}',0$ and $1$. To show that, indeed, $K$ is closed under lattice operations, we compute $x\lor^{\alga}y$ and $x\land^{\alga}y$, for any pair $\{x,y\}\subseteq K$. We start from the assumption that ($\star$) elements in $K$ are pairwise distinct. Therefore, one must have:
\begin{enumerate}[(a.)]
\item $a\parallel b',(a\land a')\lor(b\land b'),b\land b',b'\lor(a\land a'),a'\land b,a'\land(b\lor b'),a'$ and $a'\parallel b,(a\lor a')\land(b\lor b'),b\lor b',b\land(a\lor a'),a\lor b',a\lor(b\land b')$;
\item $b'\parallel a\land a',(a\land a')\lor(b\land b'),a\lor(b\land b'),a'\land b,b\land(a\lor a'),b$ and $b\parallel a\lor a',(a\lor a')\land(b\lor b'),a'\lor(b\land b'),a\lor b',b'\lor(a\land a')$;
\item $a\land a'\parallel b\land b'$ and $a\lor a'\parallel b\lor b'$;
\item $b'\lor(a\land a')\parallel a'\land b,a\lor(b\land b'),b\land(a\lor a')$ and $b\land(a\lor a')\parallel a\lor b',a'\land(b\lor b')$;
\item $a\lor(b\land b')\parallel a'\land b,a'\land(b\lor b')$ and $a'\land(b\lor b')\parallel a\lor b'$;
\item $a'\land b\parallel a\lor b'$.
\end{enumerate}
Let us prove that ${K}$ forms a subalgebra having the shape depicted in Figure \eqref{fig: kln}:
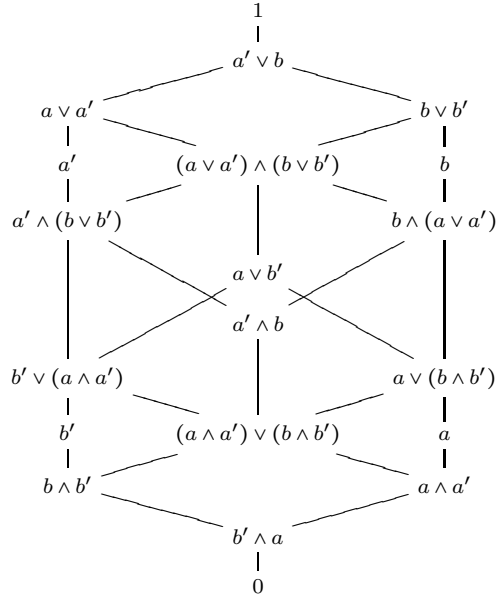
\begin{equation}\label{fig: kln}
{\tiny\xymatrix @C=1pc @R=.5pc{
&1\ar@{-}[d]&\\
&a'\lor b\ar@{-}[dl]\ar@{-}[dr]&\\
a\lor a'\ar@{-}[dr]\ar@{-}[d]&&b\lor b'\ar@{-}[dl]\ar@{-}[d]\\
a'\ar@{-}[d]&(a\lor a')\land(b\lor b')\ar@{-}[dd]\ar@{-}[dr]
\ar@{-}[dl]&b\ar@{-}[d]\\
a'\land(b\lor b')\ar@{-}[ddd]\ar@{-}[ddr]&&b\land(a\lor a')
\ar@{-}[ddd]\ar@{-}[ddl]\\
&a\lor b'\ar@{-}[ddl]\ar@{-}[ddr]&\\
&a'\land b\ar@{-}[dd]&\\
b'\lor(a\land a')\ar@{-}[d]\ar@{-}[dr]&&a\lor(b\land b')\ar@{-}
[d]\ar@{-}[dl]\\
b'\ar@{-}[d]&(a\land a')\lor(b\land b')\ar@{-}[dl]\ar@{-}
[dr]&a\ar@{-}[d]\\
b\land b'\ar@{-}[dr]&&a\land a'\ar@{-}[dl]\\
&b'\land a\ar@{-}[d]&\\
&0&
}}
\end{equation}
To this aim, we confine ourselves to show all and only non-trivial identities.\\
One has $(a\land a')\land(b\land b')=(a\land b)\land(a'\land b')=a\land b'$, since $'$ is antitone. Moreover, $(a\land a')\lor(b\land b')\lor a=a\lor(b\land b')$ holds by absorption. Similarly, one has $b'\lor(a\land a')\lor(b\land b')=b'\lor(a\land a')$.\\
Also, $a\land a'\leq((a\land a')\lor(b\land b'))\land a\leq a\land (a'\land b)=a\land a'$. Analogously, one proves $b'\land((a\land a')\lor(b\land b'))=b\land b'$. Since De Morgan's laws hold, we have also $a'\land((a\lor a')\land(b\lor b'))=a'\land(b\lor b')$, $b\land((a\lor a')\land(b\lor b'))=b\land(a\lor a')$, $a'\lor((a\lor a')\land(b\lor b'))=a\lor a'$, $b\lor((a\lor a')\land(b\lor b'))=b\lor b'$, $(b\lor b')\lor(a\lor a')=b\lor a'$.\\
Now, $b'\lor a\lor(b\land b')=b'\lor a$ follows by absorption, while $b'\lor(a'\land b)=a'\land(b\lor b')$ and $b'\lor(b\land(a\lor a'))=(a\lor a')\land (b\lor b')$ follow from \eqref{sp} and the fact that $b'\lor(b\land(a\lor a'))=b'\lor a\lor (a'\land b)=(a\lor a')\land (b\lor b')$, by SP2. The proof of $a\lor(b'\lor(a\land a'))=b'\lor a$, $a\lor(a'\land b)=b\land(a\lor a')$ and  $a\lor(a'\land(b\lor b'))=(a\lor a')\land(b\lor b')$ is similar.\\ 
By De Morgan's laws, we obtain $a'\land(a\lor b')=b'\lor(a\land a'), a'\land(a\lor(b\land b'))=(a'\land a')\lor(b\land b'), a'\land b\land(a\lor a')=a'\land b$, and $b\land(a\lor b')=a\lor(b\land b'), b\land(b'\lor(a\land a'))=(b\land b')\lor(a\land a'), b\land a'\land(b\lor b')=a'\land b$.\\
Finally, let us focus on the interval $[(a\land a')\lor(b\land b'),(a\lor a')\land(b\lor b')]$. $(b'\lor(a\land a'))\land(a'\land b)=a'\land(b'\lor a)\land b=(b\land b')\lor(a\land a')$ follows by \eqref{sp} and SP2. Furthermore, $b'\lor(a\land a')\lor(a'\land b)=b'\lor(a'\land b)=a'\land (b\lor b')$. Similarly, we prove $a\lor(b\land b')\lor(a'\land b)=b\land(a\lor a')$ and $(a\lor(b\land b'))\land(a'\land b)=(a\land a')\lor(b\land b').$ Also, $b'\lor(a\land a')\lor(b\land(a\lor a'))=b'\lor(b\land(a\lor a'))=(b'\lor b)\land(a'\lor a)$ follows again by \eqref{sp} and SP2.\\ 
By De Morgan's laws we have $(b'\lor(a\land a'))\land b\land(a\lor a')=(a\land a')\lor(b\land b')$. An analogous proof yields $a\lor(b\land b')\lor(a'\land(b\lor b'))=(a\lor a')\land(b\lor b')$ and $(a\lor(b\land b'))\land(a'\land(b\lor b'))= (a\land a')\lor(b\land b')$. Furthermore, we have $(a\land a')\lor(b\land b')\leq(b'\lor(a\land a'))\land(a\lor(b\land b'))\leq b'\lor(a\land a'))\land b\land(a\lor a')=(a\land a')\lor(b\land b')$ and so, by De Morgan's laws, $(a'\land(b\lor b'))\lor(b\land (a\lor a'))=(a\lor a')\land(b\lor b')$. $b'\lor(a\land a')\lor a\lor(b\land b')=b'\lor a$ and $b\land(a\lor a')\land a'\land(b\lor b')=a'\land b$ follow by absorption.\\ Finally, $(a'\land b)\land(a\lor b')=(a\land a')\lor(b\land b')$ holds by SP2 and so its De Morgan's dual holds as well.\\ Therefore, $K$ is the universe of a sub-algebra $\K$ of $\alga$ which is distributive. By Zorn's lemma, $\K$ can be extended to some $\K_{i}\in\mathfrak{B}(\alga)$ such that $a\leq^{\K_{i}}b$.\\ 
By virtue of the above proof, if $(\star)$ does not hold, then our conclusion follows as well, since it can be seen that $K$ is the universe of a sub-algebra of $\alga$ which is a \emph{homomorphic image} of the pseudo-Kleene lattice depicted in Fig. \eqref{fig: kln} and so again distributive. 
\end{proof} 

It is well known that, given a separable Hilbert space $\mathcal{H}$, $\mathrm{Sh}(\mathbf{E}(\mathcal{H}))$ coincides with the set $\Pi(\mathcal{H})$ of projection operators over $\mathcal{H}$ \cite[Corollary 4.2]{DeGroote}. Therefore, in view of Section \ref{sec:introduction}, $\mathrm{Sh}(\mathbf{E}(\mathcal{H}))$ forms the universe of a complete orthomodular sub-algebra of $\mathbf{E}(\mathcal{H})$. Now, it naturally raises the question if this still holds once we consider an arbitrary sp-orthomodular lattice. Unfortunately, the answer is negative, as witnessed by the next example (see also \cite[p. 1150]{GiuLePa}).
\begin{example}\label{example:psklnotshsubalg}Consider the pseudo-Kleene lattice $\algb$ depicted below.
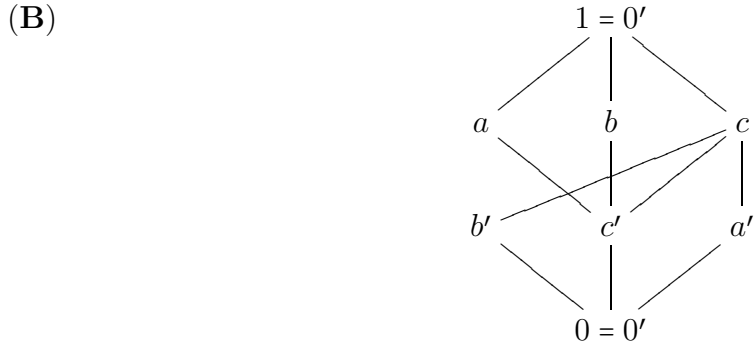
\begin{equation}
\xymatrix{
&1=0'\ar@{-}[d]\ar@{-}[dr]\ar@{-}[dl]&\\
a\ar@{-}[dr]&b\ar@{-}[d]&c\ar@{-}[dll]\ar@{-}[dl]\ar@{-}[d]\\
b'\ar@{-}[dr]&c'\ar@{-}[d]&a'\ar@{-}[dl]\\
&0=0'&
}\tag{$\algb$} 
\end{equation}
A routine check shows that $\algb$ is sp-orthomodular. However,  $a,b\in\mathrm{Sh}(\algb)$ but $a\land b\notin\mathrm{Sh}(\algb)$.
\end{example}
Of course, the class $\mathcal{C}$ of pseudo-Kleene lattices whose sharp members form a sub-algebra is a sub-quasi-variety of $\mathcal{PKL}$, since it can be axiomatised by the following quasi-equation:
\[x\land x'=0\text{ and }y\land y'=0\text{ imply }(x\land y)\land(x\land y)'=0.\label{ses}\tag{A}\]
However, $\mathcal{C}$ is \emph{proper}, since \eqref{ses} cannot be expressed equationally. For let us consider the following pseudo-Kleene (sp-orthomodular) lattice $\algc$:
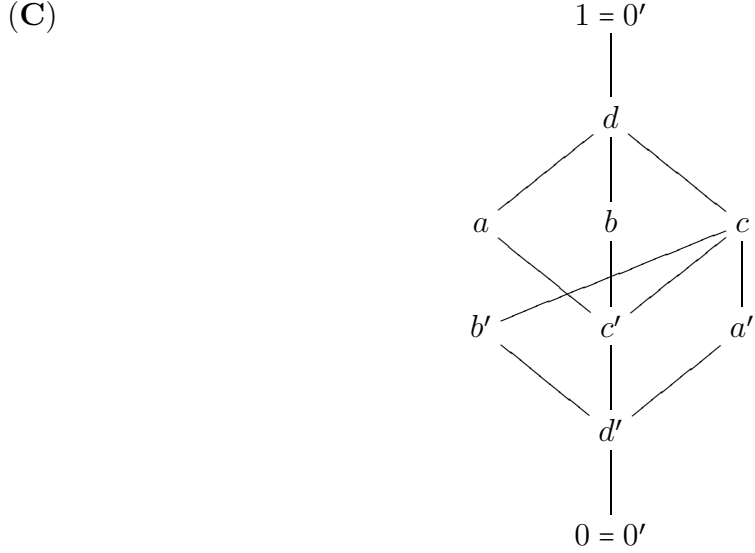
\begin{equation}
\xymatrix{
&1=0'\ar@{-}[d]&\\
&d\ar@{-}[dr]\ar@{-}[dl]\ar@{-}[d]&\\
a\ar@{-}[dr]&b\ar@{-}[d]&c\ar@{-}[dll]\ar@{-}[dl]\ar@{-}[d]\\
b'\ar@{-}[dr]&c'\ar@{-}[d]&a'\ar@{-}[dl]\\
&d'\ar@{-}[d]&\\
&0=0'&
}\tag{$\algc$} 
\end{equation}
Clearly, $\algc$ vacuously satisfies \eqref{ses}. However, if one considers the equivalence $\theta\subseteq B^{2}$ such that $[0]_{\theta}=\{0,d'\}$, $[1]_{\theta}=\{1,d\}$, and $[x]_{\theta}=\{x\}$, for any other $x\notin\{0,1,d,d'\}$, it is easily seen that $\theta$ is a congruence. Moreover, the quotient $\algb/\theta$ does not satisfy \eqref{ses} (cf. Example \ref{example:psklnotshsubalg}). Therefore, $\mathcal{C}$ is not closed under quotients. Also, it can be proven that the variety $\mathcal{V}(\mathcal{C})$ generated by $\mathcal{C}$ in $\mathcal{PKL}$ coincides with $\mathcal{PKL}$ itself.
\begin{lemma}$\mathcal{V}(\mathcal{C})=\mathcal{PKL}$.
\end{lemma}
\begin{proof}To show that $\mathcal{PKL}\subseteq\mathcal{V}(\mathcal{C})$ we prove that any $\alga\in\mathcal{PKL}$ is a quotient of some $\algb\in\mathcal{C}$. This can be accomplished by mimicking the same proof strategy of \cite[Theorem 3.2(ii)]{GiuLePa}. Details are left to the reader.
\end{proof}
A similar proof yields the following 
\begin{lemma}$\mathcal{V}(\mathcal{C}\cap\SPO)=\SPO$.
\end{lemma}
However, it can be shown that sharp elements of a sp-orthomodular lattice $\alga$ always form a sub-orthomodular poset $(\mathrm{Sh}(\alga),\leq,{}',0,1)$ of $\alga$, where $\leq$ is the order inherited from $\alga$.\\
\begin{definition}Let $\alga\in\mathcal{PKL}$ and let $\algb=(A,\leq,{}',0,1)$ be a sub-bounded poset with antitone involution of $\alga$. $\algb$ is said to be a \emph{sub-orthomodular poset} of $\alga$ provided that $\algb\in\mathcal{OMP}$ and, for any $x,y\in B$, if $x\leq y'$, then $x\lor^{\algb}y=x\lor^{\alga}y$.
\end{definition}

\begin{lemma}\label{lem:shortomdpos}Let $\alga\in\mathcal{SPO}$. Then $(\mathrm{Sh}(\alga),\leq,{}',0,1)$ is a sub-orthomodular poset of $\alga$.  
\end{lemma}
\begin{proof}
To see this, just note that, $0,1\in\mathrm{Sh}(\alga)$ and, for any $x\in \mathrm{Sh}(\alga)$, $x'\in \mathrm{Sh}(\alga)$. Moreover, if $x,y\in\mathrm{Sh}(\alga)$ are such that $x\leq y'$, then one has that $(x\lor y)\land x'\land y'=((x\land x')\lor y)\land y'=(0\lor y)\land y'=0.$ So, $x\lor y\in\mathcal{S}(\alga)$. Proving that the orthomodular law holds is straightforward, and so it is left to the reader.
\end{proof}
From now on, given $\alga\in\SPO$, we denote by $\mathbf{Sh}(\alga)$ the orthomodular poset $(\mathrm{Sh}(\alga),\leq,',0,1)$.

In what follows we introduce an important concept that will be useful for investigating properties of sp-orthomodular lattice: localizers.

Let $\alga\in\mathcal{PKL}$. For any $x,y\in A$, we set
\[0_{x,y}:=(x\land x')\lor(y\land y'),\ 1_{x,y}:=(x\lor x')\land(y\lor y')\quad\text{and}\quad \pi_{x}(y)=(y\land(x\lor x'))\lor(x\land x').\]
Given a poset $(P,\leq)$, let us denote, for any $x,y\in P$ such that $x\leq y$, \[[x,y]=\{z\in P:x\leq z\leq y\}.\] The following lemma is immediate.
\begin{lemma}Let $\alga\in\mathcal{PKL}$. Then \[\mathbf{Local}^{\alga}(x,y)=([0_{x,y},1_{x,y}],\land,\lor,{}',0_{x,y},1_{x,y}),\] where $\land,\lor,'$ are inherited from $\alga$, is a pseudo-Kleene lattice. Moreover, $\pi_{x}(y)=\pi_{0_{x,y}}(x)$ and $\pi_{x}(y)=\pi_{0_{x,y}}(y)$. 
\end{lemma}
From now on, for any $\alga\in\mathcal{PKL}$, and $x,y\in A$, $\mathbf{Local}^{\alga}(x,y)$ will be called the \emph{localizer} of $x$ and $y$ in $\alga$. Whenever the structure $\alga$ with respect to which the localizer is considered will be clear from the context, we will write simply $\mathbf{Local}(x,y)$ in place of $\mathbf{Local}^{\alga}(x,y)$.
\begin{lemma}\label{lem: gigino}Let $\alga\in\mathcal{PKL}$. Then $\alga$ satisfies (SP2) if and only if, for any $x,y\in A$ such that $x\leq y$, the following hold:
\begin{enumerate}
\item  $\pi_{0_{x,y}}(x),\pi_{0_{x,y}}(y)\in\mathrm{Sh}(\mathbf{Local}(x,y))$;
\item $(\mathrm{Sh}(\mathbf{Local}(x,y)),\leq,{}',0_{x,y},1_{x,y})$ is an orthogonal poset.
\end{enumerate}
\end{lemma}
\begin{proof}
Concerning the only-if direction, note that if $x\leq y$, then $\pi_{0_{x,y}}(x)\land\pi_{0_{x,y}}(x)'= (x\lor(y\land y'))\land(x'\land(y\lor y'))=((x\land x')\lor((y\land y')\land(y\lor y')))=(x\land x')\lor(y\land y')$. Moreover, $\pi_{0_{x,y}}(y)\land\pi_{0_{x,y}}(y)'=y\land(x\lor x')\land(y'\lor(x\land x'))=(y\land y')\land((x\land x')\land(x\lor x'))=0_{x,y}$, since $x\leq y$ implies $y'\leq x\lor x'$. Now, let $z,u\in\mathrm{Sh}(\mathbf{Local}(x,y))$. If $z\leq u'$, then $(z \lor u)\land(z'\land u')=(z \lor u'')\land(z'\land u')=(z\land z')\lor(u\land u')=(x\land x')\lor(y\land y').$ So $z\lor u\in \mathrm{Sh}(\mathbf{Local}(x,y))$. Conversely, if $x\leq y$, by (1) one has that $x\lor(y\land y'),y'\lor(x\land x')\in\mathrm{Sh}(\mathbf{Local}(x,y))$. By (2), since $x\lor(y\land y')\leq (y'\lor(x\land x'))'$, it follows that $x\lor(y\land y')\lor y'\lor(x\land x')\in\mathrm{Sh}(\mathbf{Local}(x,y)).$ Therefore, one has $x\lor y'\in\mathrm{Sh}(\mathbf{Local}(x,y))$, namely $(x\lor y')\land(x'\land y)=0_{x,y}$.
\end{proof}


\section{On forbidden configurations and the theory of commutativity}\label{sec:forbconfandcomm}
In this section we provide a forbidden configuration theorem for sp-orthomodular lattices with the aim of providing an order-theoretical explanation of the sp-orthomodular law. Furthermore, in the last part of the section we will deal with the theory of commutativity for \emph{modular} pseudo-Kleene lattices to provide (equational) sufficient and necessary conditions under which a pair of elements generate a Kleene-subalgebra. As it will be clear, although they share certain important features, orthomodular lattices and sp-orthomodular lattices yield theories which are different in some of their most crucial aspects. 
\subsection{A forbidden configuration theorem}
There is a long and time-honoured tradition that aims at characterising subvarieties
of varieties of ordered algebras in terms of ``forbidden configurations'', which dates back 
to Dedekind’s celebrated result to the effect that the distributive sub-variety of the variety of lattices is the one whose members do not contain as sub-algebras $\mathbf{M}_{3}$ or $\mathbf{N}_{5}$, while the modular sub-variety is the one whose members do not contain $\mathbf{N}_{5}$. Analogous results  appear in the theory of ortholattices. In fact, in view of Lemma \ref{lem: tameorthom}, it is easily seen that $\mathcal{OML}$ coincides with the class of ortholattices which do not contain a sub-algebra isomorphic to $\algb_{6}$. Now, due to the apparent ``resemblances'' between $\mathcal{SPO}$ and $\mathcal{OML}$, it naturally raises the question if the former class may be provided with a similar characterization.

In what follows we prove a forbidden configuration theorem to characterize sp-orthomodular lattices as pseudo-Kleene lattices in which certain (quotients of) subalgebras cannot occur. Specifically, we will show that sp-orthomodular lattices coincide with the class of pseudo-Kleene lattices whose members cannot contain a subalgebra $\algb$ orthoisomorphic to $\algb_{6}$, or to $\algb_{8}$, or such that, for some $\theta\in\mathrm{Con}\algb$, $\algb/\theta$ is orthoisomorphic either to $\algb_{8}^{*}$, or to $\algb_{10}$ depicted below:

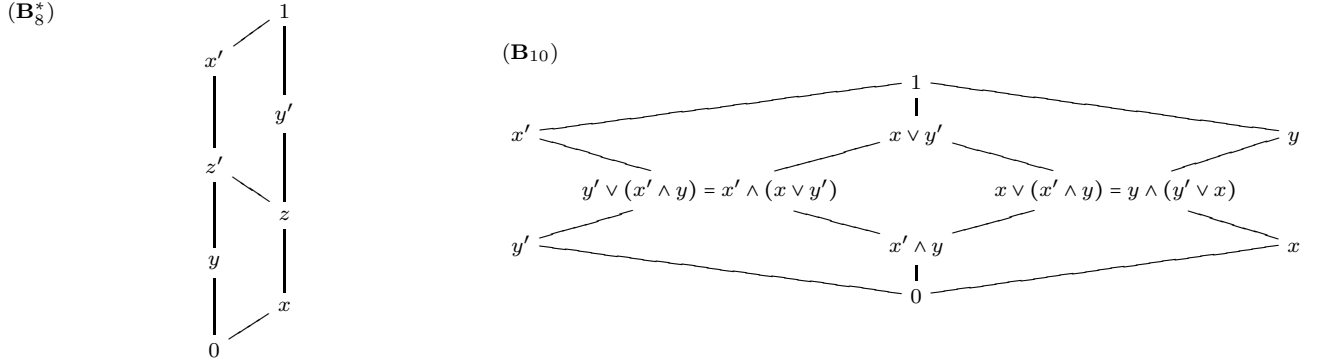
\begin{minipage}{.4\linewidth}
\begin{equation}
\tiny{\xymatrix @C=1pc @R=.5pc{
&1\ar@{-}[dl]\ar@{-}[dd]\\
x'\ar@{-}[dd]&\\
&y'\ar@{-}[dd]\\
z'\ar@{-}[dd]\ar@{-}[dr]&\\
&z\ar@{-}[dd]\\
y\ar@{-}[dd]&\\
&x\ar@{-}[dl]\\
0&
}}\tag{$\algb_{8}^{*}$}
\end{equation}
\end{minipage}
\begin{minipage}{.5\linewidth}
\begin{equation}
\tiny{\xymatrix @C=1pc @R=.5pc{
&&1\ar@{-}[dll]\ar@{-}[d]\ar@{-}[drr]&&\\
x'\ar@{-}[dr]&&x\lor y'\ar@{-}[dr]\ar@{-}[dl]&&y\ar@{-}[dl]\\
&y'\lor(x'\land y)=x'\land(x\lor y')\ar@{-}[dl]\ar@{-}[dr]&&x\lor(x'\land y)=y\land(y'\lor x)\ar@{-}[dl]\ar@{-}[dr]&\\
y'\ar@{-}[drr]&&x'\land y\ar@{-}[d]&&x\ar@{-}[dll]\\
&&0&&
}}\tag{$\algb_{10}$}
\end{equation}
\end{minipage}
\begin{theorem}\label{thm: forbiddenconfigur}Let $\alga\in\mathcal{PKL}$. Then $\alga\in\mathcal{SPO}$ if and only if it does not contain a sub-algebra $\algb$ isomorphic either to $\algb_{6}$, or to $\algb_{8}$, or such that, for some $\theta\in\mathrm{Con}\algb$,    $\algb/\theta\cong\algb_{8}^{*}$ or $\algb/\theta\cong\algb_{10}$.
\end{theorem}
\begin{proof}
If $\alga$ contains a subalgebra isomorphic either to $\algb_{6}$ or to $\algb_{8}$, then, by Theorem \ref{thm: forb config}, it does not satisfy (SP1) and so $\alga\notin\SPO$. Now, suppose that $\alga$ contains a subalgebra $\algb$ such that $\algb/\theta\cong\algb_{8}^{*}$ or $\algb/\theta\cong\algb_{10}$. In both cases, it is easily verified that $\algb/\theta$ does not satisfy (SP2) and so, by Lemma \ref{eqn: charact}, it does not satisfy \eqref{sp} as well. Since \eqref{sp} is an equational condition (Remark \ref{rem: sp-equational}), and so it is preserved by subalgebras and quotients, then $\alga$ must not satisfy \eqref{sp} as well. So $\alga\notin\SPO$. Conversely, let us suppose that $\alga$ does not satisfy \eqref{sp}. By Lemma \ref{eqn: charact}, $\alga$ does not satisfy (SP1) or (SP2). If $\alga$ does not satisfy (SP1), then it must contain a subalgebra $\algb$ isomorphic to $\algb_{6}$, or to $\algb_{8}$. Therefore, we can assume without loss of generality that $\alga$ satisfies (SP1) but it does not  satisfy (SP2). We show that $\alga$ must contain a subalgebra $\algb$ having a quotient isomorphic to $\algb_{8}^{*}$, or to $\algb_{10}$. So, let $x,y\in A$ be such that $x\leq y$, but $x'\land y\land(y'\lor x) > (x'\land x)\lor (y'\land y)$. If $y'\leq y$ and $x\leq x'$, then $x'\land y\land(x\lor y')=x\lor y'=(x\land x')\lor(y\land y')$, a contradiction. So, let us assume without loss of generality that $y'\nleq y$ or $x\nleq x'$. We consider $y'\nleq y$, since the latter case can be treated symmetrically. Moreover, if $x\leq y\leq y'\leq x'$ or $y'\leq x'\leq x\leq y$, then $x,y$ generate a distributive subalgebra of $\alga$ and so (SP2) must hold, against our assumption. So, we can assume also $y\nleq y'$ and $x'\nleq x$. Also, if $x\leq y'$, then $x\leq x'$, and we have $x'\land y\land(x\lor y')=y\land y'=x\lor(y\land y')=(x\land x')\lor(y\land y')$, contradicting again our working hypothesis. Therefore, we can assume also $x\nleq y'$. The following cases must be taken into account: (a) $x,y\in\mathrm{Sh}(\alga)$; (b) $y\in\mathrm{Sh}(\alga)$ and $x\notin\mathrm{Sh}(\alga)$ (and, symmetrically, $x\in\mathrm{Sh}(\alga)$ and $y\notin\mathrm{Sh}(\alga)$); (c) $x,y\notin\mathrm{Sh}(\alga)$. In the sequel we focus on case (c) only, since the remaining ones can be treated similarly or with much more ease, so they are left to the reader.

Given the above assumptions, two situations must be considered: (1) $x\leq x'$ and (2) $x\nleq x'$.\\
(1) Note that $(x\land x')\lor(y\land y')=x\lor(y\land y')$ as well as $x'\land y\land(y'\lor x)=y\land(x\lor y')$. Therefore, we have $x\lor(y\land y')< y\land(x\lor y')$. If $y'\geq y\land(x\lor y')$, then $y\land(x\lor y')=y\land y'$, and so $x\leq y'$, against our assumption. Therefore, we have $y\land (x\lor y')\parallel y'$. Now, if $y\land x'\leq y'\lor x$, then one must have $y'\lor x \leq y'\lor(y\land x')\leq y'\lor x$, i.e. $y'\lor x= y'\lor(y\land x')$ and so $y\land(x\lor y')=x'\land y$. Therefore, $\alga$ contains the following subalgebra $\algf_{1}$:
\begin{equation}
\tiny{
\xymatrix @C=1pc @R=.5pc{
&&1\ar@{-}[d]\\
&&y\lor y'\ar@{-}[dll]\ar@{-}[dd]\\
x'\land(y\lor y')\ar@{-}[d]&&\\
y'\lor x= y'\lor(x'\land y)\ar@{-}[drr]\ar@{-}[d]&&y\ar@{-}[d]\\
y'\ar@{-}[dd]&&y\land(x\lor y')=x'\land y\ar@{-}[d]\\
&&x\lor (y\land y')\ar@{-}[dll]\\
y\land y'\ar@{-}[d]&&\\
0&&}
} \tag{$\algf_{1}$}
\end{equation}
Upon taking the equivalence $\theta$ such that $[0]_{\theta}=\{0,y\land y'\}$, $[1]_{\theta}=\{1,y\lor y'\}$ and $[z]_{\theta}=\{z\}$ for any other $z\notin\{0,1,y\land y',y\lor y'\}$, one has that $\algf_{1}/\theta\cong\algb_{8}^{*}$. So let us consider the case $x'\land y\nleq x\lor y'$, i.e. $x'\land y\parallel y'\lor x$. Note that $y\parallel x'\land(y\lor y')$, otherwise either $y'\leq y$ or $y\leq x'$, which is impossible. Therefore, the following subalgebra $\algf_2$ obtains:
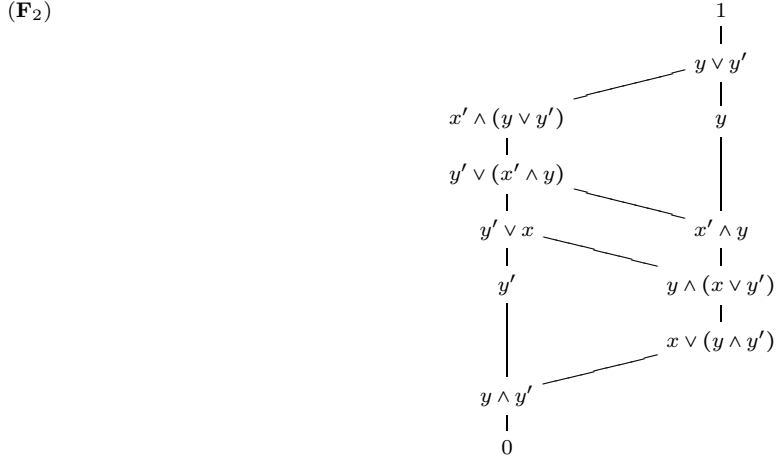
\begin{equation}
\tiny{
\xymatrix @C=1pc @R=.5pc{
&&1\ar@{-}[d]\\
&&y\lor y'\ar@{-}[dll]\ar@{-}[d]\\
x'\land(y\lor y')\ar@{-}[d]&&y\ar@{-}[dd]\\
y'\lor(x'\land y)\ar@{-}[drr]\ar@{-}[d]&&\\
y'\lor x\ar@{-}[d]\ar@{-}[drr]&&x'\land y\ar@{-}[d]\\
y'\ar@{-}[dd]&&y\land(x\lor y')\ar@{-}[d]\\
&&x\lor (y\land y')\ar@{-}[dll]\\
y\land y'\ar@{-}[d]&&\\
0&&}
} \tag{$\algf_{2}$}
\end{equation}
Let us consider the equivalence relation $\theta\subseteq F_{2}^{2}$ such that $[y'\lor x]_{\theta}=\{y'\lor(x'\land y), y'\lor x\}$, $[y\land x']_{\theta}=\{y\land(x\lor y'), y\land x'\}$, $[0]_{\theta}=\{0,y\land y'\}$, $[1]_{\theta}=\{1,y\lor y'\}$ and $[z]_{\theta}=\{z\}$, for any other $z\notin\{y'\lor(x'\land y), y'\lor x,y\land(x\lor y'), y\land x',0,1,y\land y',y\lor y'\}$. A close look shows that $\theta$ is indeed a congruence and $\algf_{2}/\theta\cong\algb_{8}^{*}$.

(2) Let us assume that $x\nleq x'$. We consider the following further sub-cases:
\begin{itemize}
\item[(2.1)] $x'\land y\leq y'\lor x$. Note that $x'\land y\leq y'\lor(x'\land y)\leq x'\land(x\lor y')\leq y'\lor x$. Moreover, $x'\land y\ne y'\lor(x'\land y)$ and $x'\land(x\lor y')\ne y'\lor x$. If $x'\land (x\lor y')\ne y'\lor(x'\land y)$, then $\alga$ contains the following subalgebra:
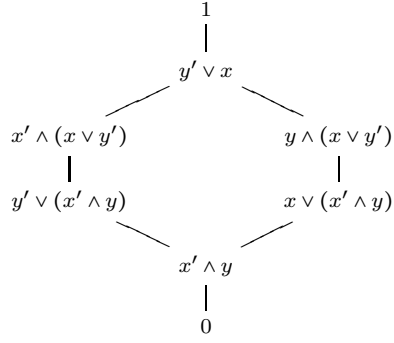
\begin{equation}
\tiny{\xymatrix @C=1pc @R=.8pc {
&1\ar@{-}[d]&\\
&y'\lor x\ar@{-}[dl]\ar@{-}[dr]&\\
x'\land (x\lor y')\ar@{-}[d]&&y\land (x\lor y')\ar@{-}[d]\\
y'\lor(x'\land y)\ar@{-}[dr]&&x\lor(x'\land y)\ar@{-}[dl]\\
&x'\land y\ar@{-}[d]&\\
&0&}}
\end{equation}
In particular, $x'\land (x\lor y')\land y\land(y'\lor x)=x'\land y$ and elements in the subalgebra but $x'\land y$, $y'\lor x$, $0$ and $1$ must be incomparable. Therefore, we would have a subalgebra isomorphic to $\algb_8$, a contradiction. We conclude ($\star$) $x'\land(y'\lor x)=y'\lor(x'\land y)$. Now, if $y'\lor(x\land x')\geq x'\land y$ and $y\land y'\ne(x\land x')\lor(y\land y')$, then $y'\lor(x\land x')=y'\lor(x'\land y)$ and we have the following  subalgebra $\algf_3$ of $\alga$:
\begin{equation}
\tiny{
\xymatrix @C=1pc @R=.5pc{
&&1\ar@{-}[d]\\
&&y\lor y'\ar@{-}[dll]\ar@{-}[d]\\
(x'\lor x)\land(y\lor y')\ar@{-}[d]&&y\ar@{-}[dd]\\
y'\lor x\ar@{-}[drr]\ar@{-}[d]&&\\
x'\land(y'\lor x)=y'\lor (x\land x')=y'\lor(x'\land y)\ar@{-}[d]\ar@{-}[drr]&&x\lor(y\land x')=y\land (x\lor x')=y\land(x\lor y')\ar@{-}[d]\\
y'\ar@{-}[dd]&&x'\land y\ar@{-}[d]\\
&&(x\land x')\lor(y\land y')\ar@{-}[dll]\\
y\land y'\ar@{-}[d]&&\\
0&&}
} \tag{$\algf_{3}$}
\end{equation}
\noindent Reasoning as in case (1), one obtains a congruence $\theta$ over $\algf_3$ such that $\algf_{3}/\theta\cong\algb_{8}^{*}$. Arguing symmetrically, if $x\lor (y\land y')\geq x'\land y$ and $x\land x'\ne (x\land x')\lor(y\land y')$, one has that $x\lor(y\land y')=x\lor(x'\land y)$ and we obtain a subalgebra $\algf_4$ of $\alga$ having the following shape
\begin{equation}
\tiny{
\xymatrix @C=1pc @R=.5pc{
&&1\ar@{-}[d]\\
&&x\lor x'\ar@{-}[dll]\ar@{-}[d]\\
(x'\lor x)\land(y\lor y')\ar@{-}[d]&&x'\ar@{-}[dd]\\
y'\lor x\ar@{-}[drr]\ar@{-}[d]&&\\
x\lor(y\land y')=x\lor (x'\land y)\ar@{-}[d]\ar@{-}[drr]&&x'\land(y\lor y')=x'\land (x\lor y')\ar@{-}[d]\\
x\ar@{-}[dd]&&x'\land y\ar@{-}[d]\\
&&(x\land x')\lor(y\land y')\ar@{-}[dll]\\
x\land x'\ar@{-}[d]&&\\
0&&}
} \tag{$\algf_{4}$}
\end{equation}
Reasoning as above, it is not difficult to find a congruence $\theta$ over $\algf_{4}$ such that $\algf_{4}/\theta\cong\algb_{8}^{*}$. Therefore, we can safely assume that (a) $(x'\land x)\lor(y'\land y)= y\land y'$ or $y'\lor(x\land x')\ngeq x'\land y$; and (b) $(x'\land x)\lor(y'\land y)= x\land x'$ or $x\lor(y\land y')\ngeq x'\land y$.
 We consider the following subcases:\\
(2.1.1) $(x'\land x)\lor(y'\land y)= y\land y'$ or $(x'\land x)\lor(y'\land y)= x\land x'$. We analyze the case $(x'\land x)\lor(y'\land y)= y\land y'$, since the latter can be treated symmetrically. We have $y\land y'\geq x\land x'=y'\land x$. Therefore $y'\lor (x\land x') = y'$. If $x\lor (y\land y')\geq x'\land y$, then, by hypothesis (b), $x\land x'=y\land y'$, and then, upon taking into account ($\star$), $\alga$ contains the following subalgebra $\algf_{5}$
\begin{equation}
\tiny{\xymatrix @C=1pc @R=.5pc{
&&1\ar@{-}[d]&&\\
&&x\lor x'=y\lor y'\ar@{-}[dll]\ar@{-}[d]\ar@{-}[drr]&&\\
x'\ar@{-}[dr]&&x\lor y'\ar@{-}[dr]\ar@{-}[dl]&&y\ar@{-}[dl]\\
&y'\lor(x'\land y)\ar@{-}[dl]\ar@{-}[dr]&&x\lor(x'\land y)\ar@{-}[dl]\ar@{-}[dr]&\\
y'\ar@{-}[drr]&&x'\land y\ar@{-}[d]&&x\ar@{-}[dll]\\
&&x\land x'=y\land y'\ar@{-}[d]&&\\
&&0&&
}}\tag{$\algf_{5}$}
\end{equation}
Taking the equivalence $\theta\subseteq F_{5}^{2}$ such that $[0]_{\theta}=\{0,y\land y'\}$, $[1]_{\theta}=\{1,y\lor y'\}$, $[z]_{\theta}=\{z\}$, for any other $z\notin\{1,0,y\land y',y\lor y'\}$, it can be verified that $\theta$ is a congruence and $\algf_{5}/\theta$ is indeed isomorphic to $\algb_{10}$. If  $x\lor(y\land y')\ngeq x'\land y$, then we consider the following subcases:\\
(2.1.1.1) $x\land x'=y\land y'$. But then, $x=x\lor(y\land y')$ and so, reasoning as above, it follows that $\alga$ contains a subalgebra having a quotient isomorphic to $\algb_{10}$.

\noindent(2.1.1.2) $x\land x'\ne y\land y'$. If $x'\land(x\lor(y\land y'))>y\land y'$, we have the following subalgebra $\algf_{6}$: 
\begin{equation}
\tiny{
\xymatrix @C=1pc @R=.5pc{
&&1\ar@{-}[d]\\
&&x\lor x'\ar@{-}[dll]\ar@{-}[d]\\
y\lor y'\ar@{-}[d]&&x'\ar@{-}[dd]\\
x\lor(x'\land(y\lor y'))\ar@{-}[drr]\ar@{-}[d]&&\\
x\lor (y\land y')\ar@{-}[d]\ar@{-}[drr]&&x'\land(y\lor y')\ar@{-}[d]\\
x\ar@{-}[dd]&&x'\land(x\lor(y\land y'))\ar@{-}[d]\\
&&(y\land y')\ar@{-}[dll]\\
x\land x'\ar@{-}[d]&&\\
0&&}
} \tag{$\algf_{6}$}
\end{equation}
In particular, note that $x\lor(x'\land(y\lor y'))\ne x\lor (y\land y')$ since $x\lor(y\land y')\ngeq x'\land y$. Therefore, upon taking the congruence $\rho=\theta(x\lor(x'\land(y\lor y')),x\lor (y\land y'))\lor\theta(x\land x',0)$ over $\algf_6$, it is easily verified that $\algf_{6}/\rho\cong\algb_{8}^{*}$. If $x'\land(x\lor(y\land y'))=y\land y'$, then we have the following subalgebra $\algf_{7}$:
\begin{equation}
\tiny{\xymatrix @C=1pc @R=.5pc{
&&1\ar@{-}[d]&&\\
&&x\lor x'=y\lor y'\ar@{-}[dll]\ar@{-}[d]\ar@{-}[drr]&&\\
x'\land(y\lor y')\ar@{-}[dr]&&x\lor y'\ar@{-}[dr]\ar@{-}[dl]&&y\ar@{-}[dl]\\
&y'\lor(x'\land y)\ar@{-}[dl]\ar@{-}[dr]&&x\lor(x'\land y)\ar@{-}[dl]\ar@{-}[dr]&\\
y'\ar@{-}[drr]&&x'\land y\ar@{-}[d]&&x\lor(y\land y')\ar@{-}[dll]\\
&&y\land y'\ar@{-}[d]&&\\
&&0&&
}}\tag{$\algf_{7}$}
\end{equation}
Using a customary reasoning, one obtains a quotient of $\algf_{7}$ isomorphic to $\algb_{10}$.
(2.1.2) $(x'\land x)\lor(y'\land y)\ne y\land y'$ and $(x'\land x)\lor(y'\land y)\ne x\land x'$. Therefore, in view of our assumptions, $y'\lor(x\land x')\ngeq x'\land y$ and $x\lor(y\land y')\ngeq x'\land y$. First, if $y\land (y'\lor(x\land x'))\ne (x\land x')\lor(y\land y')$ or $x'\land (x\lor(y\land y'))\ne (x\land x')\lor(y\land y')$, then in the former case one has the following subalgebra $\algf_{8}$: 
\begin{equation}
\tiny{
\xymatrix @C=1pc @R=.5pc{
&&1\ar@{-}[d]\\
&&y\lor y'\ar@{-}[dll]\ar@{-}[d]\\
(y\lor y')\land(x\lor x')\ar@{-}[d]&&y\ar@{-}[dd]\\
y'\lor(y\land(x\lor x'))\ar@{-}[drr]\ar@{-}[d]&&\\
y'\lor (x\land x')\ar@{-}[d]\ar@{-}[drr]&&y\land(x\lor x')\ar@{-}[d]\\
y'\ar@{-}[dd]&&y\land(y'\lor(x\land x'))\ar@{-}[d]\\
&&(y\land y')\lor(x\land x')\ar@{-}[dll]\\
y\land y'\ar@{-}[d]&&\\
0&&}
} \tag{$\algf_{8}$}
\end{equation}
In the latter case, the following subalgebra $\algf_{8}^{\ast}$ obtains:
\begin{equation}
\tiny{
\xymatrix @C=1pc @R=.5pc{
&&1\ar@{-}[d]\\
&&x\lor x'\ar@{-}[dll]\ar@{-}[d]\\
(y\lor y')\land(x\lor x')\ar@{-}[d]&&x'\ar@{-}[dd]\\
x\lor(x'\land(y\lor y'))\ar@{-}[drr]\ar@{-}[d]&&\\
x\lor (y\land y')\ar@{-}[d]\ar@{-}[drr]&&x'\land(y\lor y')\ar@{-}[d]\\
x\ar@{-}[dd]&&x'\land(x\lor(y\land y'))\ar@{-}[d]\\
&&(y\land y')\lor(x\land x')\ar@{-}[dll]\\
x\land x'\ar@{-}[d]&&\\
0&&}
} \tag{$\algf_{8}^{\ast}$}
\end{equation}
Reasoning as in previous cases, in both situations we have a subalgebra with a quotient isomorphic to $\algb_{8}^{*}$.
Therefore, we assume also $y\land (y'\lor(x\land x'))= (x\land x')\lor(y\land y')$ and $x'\land (x\lor(y\land y'))= (x\land x')\lor(y\land y')$. In this case, the following subalgebra $\algf_{9}$, which has indeed a quotient isomorphic to $\algb_{10}$, obtains:
\begin{equation}
\tiny{\xymatrix @C=1pc @R=.5pc{
&&1\ar@{-}[d]&&\\
&&x\lor x'=y\lor y'\ar@{-}[dll]\ar@{-}[d]\ar@{-}[drr]&&\\
x'\land(y\lor y')\ar@{-}[dr]&&x\lor y'\ar@{-}[dr]\ar@{-}[dl]&&y\land(x\lor x')\ar@{-}[dl]\\
&y'\lor(x'\land y)\ar@{-}[dl]\ar@{-}[dr]&&x\lor(x'\land y)\ar@{-}[dl]\ar@{-}[dr]&\\
y'\lor(x\land x')\ar@{-}[drr]&&x'\land y\ar@{-}[d]&&x\lor(y\land y')\ar@{-}[dll]\\
&&(y\land y')\lor(x\land x')\ar@{-}[d]&&\\
&&0&&
}}\tag{$\algf_{9}$}
\end{equation}
\item[(2.2)] $x'\land y\nleq y'\lor x$. Let us set $a':=(x'\land y)\land(y'\lor x)$. Note that $a'\leq y'\lor(x'\land y)\leq x'\land a\leq a$. Of course, $y'\lor(x'\land y)\ne a'$. Moreover,  $x'\land(y'\lor x)\ne a$, since otherwise $x'\land y\leq x\lor y'$. In view of our assumptions, if $a\land x'\ne y'\lor(x'\land y)$, then we would have the following subalgebra isomorphic to $\algb_{8}$:
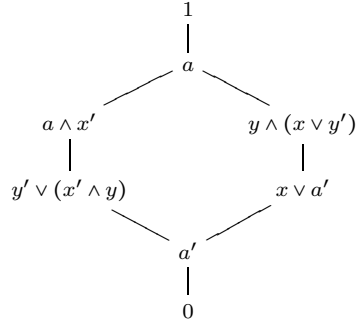
\begin{equation}
\tiny{\xymatrix @C=1pc @R=.8pc {
&1\ar@{-}[d]&\\
&a\ar@{-}[dl]\ar@{-}[dr]&\\
a\land x'\ar@{-}[d]&&y\land (x\lor y')\ar@{-}[d]\\
y'\lor(x'\land y)\ar@{-}[dr]&&x\lor a'\ar@{-}[dl]\\
&a'\ar@{-}[d]&\\
&0&}}
\end{equation}
which is impossible by hypothesis. Therefore, we conclude $a\land x'=y'\lor(x'\land y)$. Moreover, $a'\leq y'\lor a'\leq x'\land(y'\lor x)\leq a$ and, upon reasoning as above, one must have that $y'\lor a'=x'\land(y'\lor x)$. Now, let us distinguish the following subcases:\\
(2.2.1) $y'\lor(x\land x')=y'\lor a'$ or $x'\lor(y\land y')=x\lor a'$. We consider the former case, since the latter can be treated analogously. Note that, in particular, one has $y\land y'\ne(x\land x')\lor(y\land y')$, otherwise $y'\lor a'=y'$, which is impossible. In this situation, we have the following subalgebra $\algf_{10}$ of $\alga$:
\begin{equation}
\tiny{
\xymatrix @C=1pc @R=.5pc{
&&1\ar@{-}[d]\\
&&y\lor y'\ar@{-}[dll]\ar@{-}[d]\\
(y\lor y')\land(x\lor x')\ar@{-}[d]&&y\ar@{-}[dd]\\
a\ar@{-}[drr]\ar@{-}[d]&&\\
y'\lor a'\ar@{-}[d]\ar@{-}[drr]&&y\land a\ar@{-}[d]\\
y'\ar@{-}[dd]&&a'\ar@{-}[d]\\
&&(y\land y')\lor(x\land x')\ar@{-}[dll]\\
y\land y'\ar@{-}[d]&&\\
0&&}
} \tag{$\algf_{10}$}
\end{equation}
Applying a customary reasoning, we obtain a congruence $\theta$ over $\algf_{10}$ such that $\algf/\theta\cong\algb_{8}^{*}$.\\
(2.2.2) $y'\lor(x\land x')\ne y'\lor a'$ and $x'\lor(y\land y')\ne x\lor a'$. We consider the following subcases:\\
(2.2.2.1) $y\land(y'\lor(x\land x'))\ne (x\land x')\lor(y\land y')$ or $x'\land(x\lor(y\land y'))\ne (x\land x')\lor(y\land y')$. Again, we consider the former situation, since the latter can be treated in the same way. Indeed, it can be seen that $\alga$ contains $\algf_{8}$ and so, reasoning as in case (2.1.2) one has that $\alga$ contains a subalgebra having a quotient isomorphic to $\algb_{8}^{*}$.\\
(2.2.2.2) $y\land(y'\lor(x\land x'))= (x\land x')\lor(y\land y')$ and $x'\land(x\lor(y\land y'))= (x\land x')\lor(y\land y')$. The following subalgebra $\algf_{11}$ of $\alga$ obtains:
\begin{equation}
\tiny{\xymatrix @C=1pc @R=.5pc{
&1\ar@{-}[d]&\\
&(x\lor x')\land (y\lor y')\ar@{-}[dl]\ar@{-}[d]\ar@{-}[dr]&\\
x'\land(y\lor y')\ar@{-}[d]&a\ar@{-}[dl]\ar@{-}[dr]\ar@{-}[dd]&y\land(x\lor x')\ar@{-}[d]\\
x'\land a\ar@{-}[ddd]\ar@{-}[ddr]&&y\land a\ar@{-}[ddl]\ar@{-}[ddd]\\
&y'\lor x\ar@{-}[ddl]\ar@{-}[ddr]&\\
&x'\land y\ar@{-}[dd]&\\
y'\lor a'\ar@{-}[dr]\ar@{-}[d]&&x\lor a'\ar@{-}[dl]\ar@{-}[d]\\
y'\lor(x\land x')\ar@{-}[dr]&a'\ar@{-}[d]&x\lor(y\land y')\ar@{-}[dl]\\
&(x\land x')\lor(y\land y')\ar@{-}[d]&\\
&0&
}} \tag{$\algf_{11}$}
\end{equation}
Consider the following equivalence $\theta\subseteq F_{11}^{2}$ such that $[x'\land y]_{\theta}=\{x'\land y,a'\}$, $[y'\lor x]_{\theta}=\{y'\lor x,a\}$, $[x'\land a]_{\theta}=\{x'\land a,y'\lor a'\}$, $[y\land a]_{\theta}=\{x\lor a',y\land a\}$, $[(x\land x')\lor(y\land y')]_{\theta}=\{(x\land x')\lor(y\land y'),0\}$, $[(x\lor x')\land(y\lor y')]_{\theta}=\{(x\lor x')\land(y\lor y'),1\}$, and $[z]_{\theta}=\{z\}$, for any other $z\in F_{11}$ not mentioned above. It turns out that $\theta$ is a congruence such that $\algf_{11}/\theta$ has the following shape:
\begin{equation}
\tiny{\xymatrix @C=1pc @R=.5pc{
&&[1]_{\theta}\ar@{-}[dll]\ar@{-}[d]\ar@{-}[drr]&&\\
[x'\land(y\lor y')]_{\theta}\ar@{-}[dr]&&[x\lor y']_{\theta}\ar@{-}[dr]\ar@{-}[dl]&&[y\land(x\lor x')]_{\theta}\ar@{-}[dl]\\
&[x'\land a]_{\theta}\ar@{-}[dl]\ar@{-}[dr]&&[x\lor a']_{\theta}\ar@{-}[dl]\ar@{-}[dr]&\\
[y'\lor(x\land x')]_{\theta}\ar@{-}[drr]&&[x'\land y]_{\theta}\ar@{-}[d]&&[x\lor(y\land y')]_{\theta}\ar@{-}[dll]\\
&&[0]_{\theta}&&
}}\tag{$\algf_{11}/\theta$}
\end{equation}
So, we have $\algf_{11}/\theta\cong\algb_{10}$. This concludes the proof of the theorem.\qedhere
\end{itemize}
\end{proof}
As an immediate consequence of Theorem \ref{thm: forbiddenconfigur}, we have the following corollary.
\begin{corollary}Let $\alga\in\mathcal{PKL}$. Then $\alga\in\SPO$ if and only if it does not contain a sub-algebra $\algb$ such that, for some $\theta\in\mathrm{Con}\algb$, $\algb/\theta$ is isomorphic either to $\algb_{6}$, or to $\algb_{8}^{*}$, or to $\algb_{10}$.
\end{corollary}
Moreover, if $\mathcal{V}$ is a sub-variety of $\mathcal{PKL}$, then it is easily seen that $\mathcal{V}$ contains a member isomorphic to $\algb_{6}$ if and only if it contains a member isomorphic either to $\algb_{6}$ or to $\algb_{8}$. So, above results can be summarized as follows.
\begin{corollary}Let $\mathcal{V}$ be a sub-variety of $\mathcal{PKL}$. Then $\mathcal{V}\subseteq\mathcal{SPO}$ if and only if it does not contain a member isomorphic either to $\algb_{6}$, or to $\mathbf{B}_{8}^{*}$, or to $\algb_{10}$.
\end{corollary}

In \cite{GiuMurPa}, the problem of providing an equational basis for the class $\mathcal{D}$ of pseudo-Kleene lattices which do not have $\algb_{8}$ as a sub-algebra was addressed. In that venue (and by Theorem \ref{thm: forb config}), it was shown that $\mathcal{D}$ forms indeed a quasi-variety, since it can be axiomatised by a finite set of quasi-equations. However, \cite{GiuMurPa} proves also that $\mathcal{D}$ is proper, since it is not closed under quotients. Now, it is reasonable to ask for something weaker, namely if a \emph{largest} variety of pseudo-Kleene lattices which do not contain $\algb_{6}$ or $\algb_{8}$ as sub-algebras can be provided.
Due to its deep resemblances with $\mathcal{OML}$, $\SPO$ seems to be a good candidate for reaching the goal. However, this is not the case, as shown by the next lemma. 
\begin{lemma}$\mathcal{SPO}$ is \emph{not} the largest subvariety of $\mathcal{PKL}$ which does not contain a member isomorphic to $\algb_{6}$ or to $\algb_{8}$. 
\end{lemma}
\begin{proof}
We show that $\mathcal{V}(\algb_{8}^{*})\subseteq\mathcal{PKL}$ does not contain members isomorphic to $\algb_{8}$ or to $\algb_{6}$. First, since $\mathcal{V}(\algb_{8}^{*})$ is a variety of algebras having a lattice reduct, then it is congruence distributive. Therefore, by J\'onnson's Lemma \cite[Corollary 6.10]{BS}, if $\alga$ is a subdirectly irreducible member of $\mathcal{V}(\algb_{8}^{*})$, then $\alga\in\mathbb{H}\mathbb{S}(\algb_{8}^{*})$. A direct inspection yields that $\algb_{8}^{*}$ is subdirectly irreducible and, moreover, any of its quotients is distributive. Also, apart from $\algb_{8}^{*}$ itself, any subalgebra of $\algb_{8}^{*}$ is distributive. Therefore, any subdirectly irreducible member of $\mathcal{V}(\algb_{8}^{*})$ different from $\algb_{8}^{*}$ must be distributive as well. As a consequence, since there is no subdirectly irreducible member of $\mathcal{V}(\algb_{8}^{*})$ having $\algb_{6}$ or $\algb_{8}$ as a subalgebra, and since (SP1) is preserved by direct products and subalgebras, one has that there is no subdirect product of subdirectly irreducible members of $\mathcal{V}(\algb_{8}^{*})$ which contains a subalgebra isomorphic to $\algb_{6}$ or to $\algb_{8}$. Therefore $\mathcal{V}(\algb_{8}^{*})$ does not contain $\algb_{8}$ or $\algb_{6}$ and, since by Theorem \ref{thm: forbiddenconfigur} $\mathcal{V}(\algb_{8}^{*})\nsubseteq\mathcal{SPO}$, the desired result obtains.
\end{proof}

\subsection{Some remarks on the theory of commutativity}
As it has been recalled in Section \ref{sec:preliminaries}, the theory of commutativity plays an important role in orthomodular posets (lattices) theory. In fact, commutativity might be somehow regarded as an order-theoretical, abstract counterpart of commutativity between projection operators. Now, it naturally raises the question if such a notion can be introduced in the framework of sp-orthomodular lattices to the effect that one has a finite set of equations  $\{\epsilon_{i}(x,y)\approx\delta_{i}(x,y)\}_{i\in\{1,\dots,n\}}$ in the variables $x,y$ such that, for any $\alga\in\SPO$, and $a,b\in A$, $\alga\models\epsilon_{i}(a,b)\approx\delta_{i}(a,b)$ (for any $1\leq i\leq n$) iff $\mathbf{Sg}^{\alga}(a,b)$ is distributive. 
The next results show that, in some cases, the notion of commutativity already introduced for OMPs still serves as a sufficient and necessary condition for distributivity. 

\begin{lemma}\label{lem: pink}Let $\alga\in\SPO$. If $x,y\in \mathrm{Sh}(\alga)$ and $x\mathrm{C} y$ in $\mathbf{Sh}(\alga)$, then $x\land^{\mathbf{Sh}(\alga)} y=x\land^{\alga} y$.
\end{lemma}
\begin{proof}
To see this, suppose that $x\C{}y$ in $\mathbf{Sh}(\alga)$. Set $c:= x\land^{\mathbf{Sh}(\alga)}y$. One has $c\leq x\land^{\alga} y$. Therefore, by \eqref{sp}, one has $x\land^{\alga} y=(x\land^{\alga} y)\land (c\lor^{\alga} c')=c\lor^{\alga}((x\land^{\alga} y)\land^{\alga} c')$. Note that
\begin{align*}
y\land^{\alga} c' &= (y'\lor^{\alga} c)'\\
&= (y'\lor^{\alga} (x\land^{\mathbf{Sh}(\alga)} y))'\\
&= (y'\lor^{\mathbf{Sh}(\alga)} (x\land^{\mathbf{Sh}(\alga)} y))'\\
&= ((y'\lor^{\mathbf{Sh}(\alga)} x)\land^{\mathbf{Sh}(\alga)} (y'\lor^{\mathbf{Sh}(\alga)}y))'= y\land^{\mathbf{Sh}(\alga)}x'.
\end{align*}
Therefore, $(x\land^{\alga} y)\land^{\alga} c' = 0$. So $x\land^{\alga} y=c$.  
\end{proof}

\begin{corollary}\label{cor:ghz}Let $\alga\in\SPO$. If $x,y\in\mathrm{Sh}(\alga)$, then $\mathbf{Sg}(x,y)$ is distributive iff $x\C{}y$ in $\mathbf{Sh}(\alga)$. In particular, $\mathbf{Sg}(x,y)$ is a Boolean sub-algebra of $\alga$.
\end{corollary}
\begin{proof}
The right-to-left direction easily follows from Lemma \ref{lem: pink} upon noticing that in any orthomodular poset $\algp$, if $x,y$ commute, then they generate a Boolean sub-algebra of $\algp$ of mutually commuting elements. Conversely, note that, if $\mathbf{Sg}(x,y)$ is distributive, then a straightforward check shows that $x\land y,x\land y'\in\mathrm{Sh}(\alga)$ and $x = (x\land y)\lor(x\land y')$.
\end{proof}
As a consequence, one has that
\begin{corollary}
Let $\alga\in\mathcal{SPO}$. If $x\land x'=y\land y'$, then $\mathbf{Sg}(x,y)$ is distributive if and only if $x\mathrm{C}y$ in $\mathbf{Sh}(\mathbf{Local}(x,y))$. 
\end{corollary}
\begin{proof}
The right-to-left direction follows upon noticing that, if $\mathrm{Sg}(x,y)$ is distributive, then $x\land y, x'\land y\in\mathrm{Sh}(\mathbf{Local}(x,y))$ and $(x\land y)\lor(x'\land y)=y$. Conversely, if $x\land x'=y\land y'$, then  $\mathbf{Sg}(x,y)=(\{0,1\}\cup\mathrm{Sg}^{\mathbf{Local}(x,y)}(x,y),\land,\lor,',0,1)$, where $\land,\lor,'$ are inherited from $\alga$. Since $x\C{}y$, then $\mathbf{Sg}^{\mathbf{Local}(x,y)}(x,y)$ is distributive, by Corollary \ref{cor:ghz}. Moreover, $\mathbf{Sg}(x,y)^{\ell}$ is isomorphic to $0\oplus(\mathrm{Sg}^{\mathbf{Local}(x,y)}(x,y))^{\ell}\oplus 1$ which is distributive, by Remark \ref{rem:ordsum}.   
\end{proof}
As a further application of Corollary \ref{cor:ghz}, we obtain another characterization of sp-orthomodular lattices within $\mathcal{PKL}$.
\begin{lemma}Let $\alga\in\mathcal{PKL}$. Then $\alga\in\mathcal{SPO}$ if and only if, for any $x,y\in A$, if $x\leq y$, then $\pi_{0_{x,y}}(x),\pi_{0_{x,y}}(y)$ generate a Boolean subalgebra of $\mathbf{Local}(x,y)$.
\end{lemma}
\begin{proof}
Concerning the left-to-right direction, assume that $x\leq y$. Note that $\mathbf{Local}(x,y)$ can be regarded as an sp-orthomodular lattice as well in which $\pi_{0_{x,y}}(x)=x\lor(y\land y')$ and $\pi_{0_{x,y}}(y)=y\land(x\lor x')$ are both sharp (Lemma \ref{lem: gigino}). Since $x\lor (y\land y')\leq y\land(x\lor x')$, $x\lor (y\land y')\C{}y\land(x\lor x')$ in $\mathbf{Sh}(\mathbf{Local}(x,y))$. By Corollary \ref{cor:ghz}, they generate a Boolean sub-algebra of $\mathbf{Local}(x,y)$.\\ 
The converse easily follows upon noticing that $x\leq y$ implies $\pi_{0_{x,y}}(y)=y\land (x\lor x')$, $\pi_{0_{x,y}}(x)=x\lor(y\land y')$, and $x\lor(x'\land y)=x\lor(y\land y')\lor((x'\land (y\lor y'))\land(y\land(x\lor x')))=\pi_{0_{x,y}}(x)\lor(\pi_{0_{x,y}}(x)'\land\pi_{0_{x,y}}(y))=(\pi_{0_{x,y}}(x)\lor\pi_{0_{x,y}}(x)')\land(\pi_{0_{x,y}}(x)\lor\pi_{0_{x,y}}(y))=\pi_{0_{x,y}}(y)=y\land(x\lor x')$. 
\end{proof}
Now, one might wonder if a general notion of commutativity may be introduced for arbitrary pairs of elements of an sp-orthomodular lattice.
Indeed, in the last part of this section, we show that, at least for the variety of modular pseudo-Kleene lattices, a solution to this problem is an easy consequence of a well known result provided by B.~J\'onsson in \cite{Jonsson}. As a consequence, in view of Lemma \ref{lem:finitedimmodular}, a characterization of pairs of effects generating a Kleene sub-lattice of $\mathbf{E}(\mathcal{H})$, once $\mathcal{H}$ is assumed to be finite-dimensional, obtains.
\begin{lemma}\label{Jonsson}(\cite[Theorem 6]{Jonsson})Suppose that $\alga$ is a modular lattice and let $\algb$ and $\algc$ be distributive
sublattices of $\alga$. Then the sublattice of $\alga$ generated by $B\cup C$ is distributive if and only if, for any $\{x_{1},x_{2},x\}\subseteq B$ and $\{y_{1},y_{2},y\}\subseteq C$, it holds that
\[(x_{1}\lor x_{2})\land y=(x_{1}\land y)\lor(x_{2}\land y)\text{ and }(y_{1}\lor y_{2})\land x= (y_{1}\land x)\lor(y_{2}\land x).\] 
\end{lemma}
First, we observe the following fact.
\begin{lemma}\label{lem:auxcom2}Suppose that $\alga$ is a bounded lattice with antitone involution and let $\algb$ and $\algc$ be subalgebras of $\alga$. Then the bounded sublattice $\algd$ of $\alga^{\ell}$ generated by $B\cup C$ is closed under ${}'$. Therefore, $\algd$ is a subalgebra of $\alga$. 
\end{lemma}
\begin{proof}
Note that $D$ can be obtained as follows. Set:\[S_{0}:=B\cup C,\quad S_{i+1}:=S_{i}\cup\{a\lor b,a\land b:a,b\in S_{i}\},\text{ and } S:=\bigcup_{i\geq 0}S_{i}.\]
Clearly, $S\subseteq D$. Moreover, customary arguments yield that $S$ is closed under $1,0,\land,\lor$. So $S=D$. Therefore, in order to show that $D$ is closed under $'$, we can reason by induction on $i$. The case $i=0$ is obvious since $\algb$ and $\algc$ are subalgebras of $\alga$. Assume that the statement  holds for $i>0$. We show that it holds for $i+1$. Let $a\in S_{i+1}$. If $a\in S_{i}$, then $a'\in S_{i}$ by the induction hypothesis. Otherwise, $a=a_{1}\lor a_{2}$ or $a=a_{1}\land a_{2}$, for some $a_{1},a_{2}\in S_{i}$. We consider the former case only, since the latter can be treated dually. By Induction hypothesis, we have that $a_{1}',a_{2}'\in S_{i}$. So $(a_{1}\lor a_{2})'=a_{1}'\land a_{2}'\in S_{i+1}$.
\end{proof}
\begin{definition}\label{def:commmodular}
Let $\alga\in\mathcal{PKL}$ and $a,b\in A$. We say that $a$ \emph{commutes} with $b$, written $a\mathrm{C}b$, provided that the following hold:
\begin{enumerate}[(C1)]
\item $a\land(b\lor b')=(a\land b)\lor(a\land b')$;
\item $b\land(a\lor a')=(b\land a)\lor(b\land a')$;
\item $a\land a' = ((a\land a')\land b)\lor((a\land a')\land b').$
\end{enumerate}
\end{definition}
\begin{lemma}\label{aux:com1}Let $\alga\in\mathcal{MPKL}$. If $x\C{}y$, then the following hold:

\[x\lor(x'\land y')=(x\lor x')\land(x\lor y')\]

\end{lemma}
\begin{proof}
(1). We compute 
\begin{align*}
x\lor(x'\land y') =&\ x\lor(x'\land y')\lor(x\land x')\\
=&\ x\lor(x'\land(y'\lor(x\land x')))\\
=&\ x\lor(x'\land(y'\lor x)\land(y'\lor x'))\\
=&\ x\lor(x'\land(y'\lor x))\\
=&\ (y'\lor x)\land(x'\lor x).   
\end{align*}
 
\end{proof}
\begin{lemma}\label{lem:commutativity}Let $\alga\in\mathcal{MPKL}$. Then, for any $x,y\in A$, the following are equivalent:
\begin{enumerate}
\item $x\C{} y$;
\item $y\C{} x$;
\item $x\C{} y'$;
\item $x'\C{}y$;
\item $x\C{}y$ and $(x\lor x')\land(y\lor y')=((x\lor x')\land y)\lor((x\lor x')\land y')$.

\end{enumerate}
\end{lemma}
\begin{proof}$(1)\Leftrightarrow(2)$. We prove $(1)\Rightarrow(2)$, since the converse direction can be handled symmetrically. Note that $y\lor(x\land x'\land y')=y\lor(x\land x'\land y')\lor(x\lor x'\land y)=y\lor(x\land x').$ Moreover, one has $(x\lor y)\land(x\lor y')=y\lor(x'\land(x\lor y))=y\lor((x\land x')\lor(x'\land y))=y\lor(x\land x')$, by Lemma \ref{aux:com1}. We conclude $y\lor(x\land x'\land y')=(x\lor y)\land(y\lor x')$. So, we compute
\begin{align*}
x\lor(y\land y'\land x') =& x\lor(x\land x'\land y')\lor(y\land y'\land x')\\
=& x\lor((x'\land y')\land(y\lor(x\land x'\land y')))\\
=& x\lor((x'\land y')\land(x\lor y)\land(y\lor x'))\\
=& x\lor((x'\land y')\land(x\lor y))\\
=& (x\lor y)\land(x\lor(x'\land y'))\\
=& (x\lor y)\land(x\lor x')\land(x\lor y')\\
=& (x\lor y)\land(x\lor x'\lor y)\land(x\lor x'\lor y')\land(x\lor y')\\
=& (x\lor y)\land(x\lor y').
\end{align*}
Therefore, we have $(y\land y')\land((y\land y'\land x)\lor(y\land y'\land x'))=(y\land y')\land(y\land y')\land(x\lor(y\land y'\land x'))=(y\land y')\land(x\lor y)\land(x\lor y')=y\land y'$. So $y\land y'\leq(y\land y'\land x)\lor(y\land y'\land x')$, i.e. $y\land y' = (y\land y'\land x)\lor(y\land y'\land x')$.\\
$(1)\Leftrightarrow(3)$. Again, we prove $(1)\Rightarrow(3)$ since the converse holds symmetrically. We only need to prove $y'\land(x\lor x')=(y'\land x)\lor(y'\land x')$. Indeed, we have $y'\land(x\lor x')=y'\land(x\lor x')\land(x\lor y')=y'\land(x\lor(x'\land y'))=(x'\land y')\lor(y'\land x)$, by modularity and Lemma \ref{aux:com1}.\\
$(1)\Leftrightarrow(4)$. Just note that $x\C{}y$ iff $y\C{}x$ iff $y\C{}x'$ iff $x'\C{}y$.\\
Finally, as regards the non-trivial direction of $(1)\Leftrightarrow(5)$, note that $((x\lor x')\land y)\lor((x\lor x')\land y')=(x\lor x')\land(((x \lor x')\land y)\lor y').$ Moreover, $y'\lor((x\lor x')\land y)=y'\lor((x\lor x'\lor y')\land y)=(x\lor x'\lor y')\land(y\lor y')=y'\lor((x\lor x')\land(y\lor y'))$. Therefore, $(x\lor x')\land(((x \lor x')\land y)\lor y')=(x\lor x')\land(y'\lor((x\lor x')\land(y\lor y')))=((x\lor x')\land(y\lor y'))\lor(y'\land(x\lor x'))=(x\lor x')\land(y\lor y')$. 
\end{proof}
We are into position for stating and proving the final result of this section.
\begin{theorem}Let $\alga\in\mathcal{MPKL}$. For any $x,y\in A$, the following are equivalent:
\begin{enumerate}
\item $x\C{}y$;
\item $\mathbf{Sg}(x,y)$ is a Kleene lattice. 
\end{enumerate}
\end{theorem}
\begin{proof}$(2)\ra(1)$ is clear. Conversely, Let $a,b\in A$ such that $a\C{}b$. One has that $\mathbf{Sg}(a)=(\{a,a',a\land a',a\lor a',0,1\},\land,\lor,{}',0,1)$ and $\mathbf{Sg}(b)=(\{b,b',b\land b',b\lor b',0,1\},\land,\lor,{}',0,1)$ are distributive. We show that, for any $x,y\in\mathrm{Sg}(a)$, and any $z\in\mathrm{Sg}(b)$, one has $(x\lor y)\land z = (x\land z)\lor(y\land z)$. First, if $x=y$, then our conclusion holds trivially. Therefore, we assume also $x\ne y$. If $x$ or $y$ are in $\{0,1,a\land a',a\lor a'\}$, or $z\in\{0,1\}$, then the statement is straightforward. Therefore, let us assume that $x,y\in\{a,a'\}$ and $z\notin\{0,1\}$. We consider the following cases according to $z$.
\begin{itemize}
\item $z=b\land b'$. In this case $(a\lor a')\land (b\land b')=b\land b'=(b\land b'\land a)\lor(b\land b'\land a')$ is ensured upon noticing that, by Lemma \ref{lem:commutativity}, one has $b\C{}a$.
\item $z=b$ or $z=b'$. Then, again, the statement follows upon noticing that, by Lemma \ref{lem:commutativity}, $b\C{}a$ and $b'\C{}a$.
\item $z=b\lor b'$. In this case distributivity over $a\lor a'$ follows by Lemma \ref{lem:commutativity}(5).
\end{itemize}
Showing that any $x\in\mathrm{Sg}(a)$ distributes over any join of elements in $\mathrm{Sg}(b)$ can be handled similarly, and so it is left to the reader. By Lemma \ref{Jonsson} we conclude that $\mathrm{Sg}(b)\cup\mathrm{Sg}(a)$ generates a distributive sublattice of $\alga^{\ell}$ and so, by Lemma \ref{lem:auxcom2}, also a distributive subalgebra of $\alga$.
\end{proof}
Naturally, one might wonder if our notion of commutativity might be somehow simplified. Indeed, when dealing with orthomodular lattices, it turns out that conditions $(C1)$ and $(C2)$ are equivalent. Nevertheless, the next example shows that, in the wider framework of sp-orthomodular lattices, this no longer holds.
\begin{example}Let us consider the sp-orthomodular lattice depicted in Figure \eqref{fig:dmndfixpoint}:
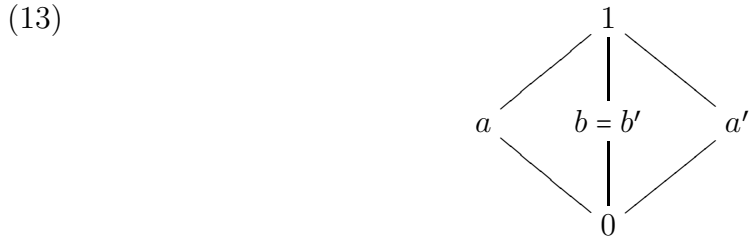
\begin{equation}
\xymatrix{
&1\ar@{-}[d]\ar@{-}[dl]\ar@{-}[dr]&\\
a\ar@{-}[dr]&b=b'\ar@{-}[d]&a'\ar@{-}[dl]\\
&0&
}\label{fig:dmndfixpoint}
\end{equation}
Note that $a\land(b\lor b')=a\land b=0=(a\land b)\lor(a\land b')$ as well as $(a\land a')=0=((a\land a')\land b)\lor((a\land a')\land b')$. However, $b=b\land(a\lor a')\ne 0=(b\land a)\lor(b\land a')$. 
\end{example}
Finally, we observe that, in the framework of sp-orthomodular lattices, the Foulis-Holland theorem (Theorem \ref{fuli}) fails. In particular, we exhibit a modular sp-orthomodular lattice $\alga$ and a triple $\{a,b,c\}\subseteq A$ such that $a\C{}b,a\C{}c,b\C{}c$ but $\{a,b,c\}$ \emph{does not} generate a distributive sub-lattice of $\alga^{\ell}$.
\begin{example}Let us consider the sp-orthomodular lattice $$\alga=(\{0,a,b,c,d,a',b',c',1\},\land,\lor,{}',0,1)$$ depited in Figure \eqref{failureFH}:
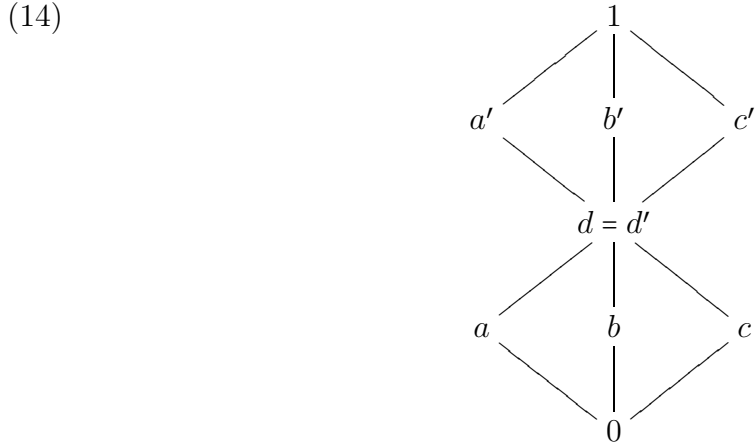
\begin{equation}
\xymatrix{
&1\ar@{-}[d]\ar@{-}[dl]\ar@{-}[dr]&\\
a'\ar@{-}[dr]&b'\ar@{-}[d]&c'\ar@{-}[dl]\\
&d=d'\ar@{-}[dl]\ar@{-}[dr]\ar@{-}[d]&\\
a\ar@{-}[dr]&b\ar@{-}[d]&c\ar@{-}[dl]\\
&0&
}\label{failureFH}
\end{equation}
An easy check shows that $a\C{}b$ and $a\C{}c$. However, $\{a,b,c\}$ does not generate a distributive sublattice of $\alga^\ell$.
\end{example}
\section{Paraconsistent partial referential matrices}\label{sec:paraconspartrefmatr}
In this section we introduce the class of paraconsistent partial referential matrices. These structures generalize \emph{partial referential matrices} (PRMs) outlined for the first time by J. Czelakowski in \cite{Czela1981b}.

Taking the cue from the intuition that quantum experimental propositions concerning the measure of an observable $A$ over a state $\psi$ should be considered \emph{meaningful} (i.e., amenable of a truth's evaluation) only in case $\psi$ is \emph{dispersion free}, i.e it is an eigenvector of $A$, \cite{Czela1981b} introduces PRMs as a general logical framework to deal with propositions which need not be everywhere defined but for certain states. A PRM is nothing but a triple $(I,\alga,\{D_{i}\}_{i\in I})$, where $I$ is non-empty set of states, $\alga$ is an ``algebra'' of partial propositions of the form $(D,f)\in\mathcal{P}(I)\times B_{2}^{D}$ assigning to any $i\in D$ a truth value in the two-element Boolean algebra $\algb_{2}$, and $D_{i}$ ($i\in I$) is the set of true propositions $(D,f)$ which are meaningful w.r.t. $i$ (i.e., $i\in D$). Interestingly enough, any atomic atomistic orthomodular poset can be regarded as the algebraic reduct of a PRM (see \cite[p. 138]{Czela1981b}). Therefore, PRMs make orthomodular quantum structures amenable of a really neat interpretation.\\
In the sequel, we introduce paraconsistent partial referential matrices as a generalization of Czelakowski's framework to deal with partial propositions whose truth values range in the set $\{0,\frac{1}{2},1\}$. In the light of results outlined in previous sections, it won't be surprising to find that any unsharp orthogonal poset $\alga$ yields the algebraic reduct of a paraconsistent partial referential matrix. Even more, if $\alga\in\SPO$, then $\alga$ turns out to be \emph{orthoisomorphic} to a pseudo-Kleene lattice of partial propositions. Consequently, a further link between logic and the algebra $\mathbf{E}(\mathcal{H})$ of effects over a separable Hilbert space $\mathcal{H}$ is established.

Let us start with the main concept we are dealing with in this section. 
\begin{definition}\label{def: parparrefmatr}A \emph{paraconsistent partial referential matrix} is a structure $\mathcal{A}=(I,\alga,\{D_{i}\}_{i\in I})$ such that:
\begin{enumerate}
\item $I$ is a set of indexes;
\item $\alga=({A,\mathbb{C},\neg,\lor,\mathbb{O},\mathbb{I}})$, called the \emph{algebraic reduct} of $\mathcal{A}$, is such that:
\begin{enumerate}[(i)]
\item $A$ consists of pairs $(C,f)$, where $C\subseteq I$ and $f\in K_{3}^{C}$.
\item $\mathbb{C}\subseteq A^{2}$ is a reflexive and symmetric binary relation of \emph{commeasurability} such that for, any $(C,f),(D,g)\in A$, $(C,f)\mathbb{C}(D,g)$ implies $C\cap D\ne \emptyset$; 
\item  $\mathbb{I}=(I,\chi_{I})$ and $\mathbb{O}=(I,\chi_{\emptyset})$;
\item If $(C,f)\in A$, then $(C,\neg f)\in A$, where, for any $x\in C$, $\neg f(x)=\neg(f(x))$ in $\K_{3}$;
\item If $(C,f)\mathbb{C}(D,g)$, then there exists a \emph{unique} $(Z,h)$ such that $D\cap C\subseteq Z$ and, for any $x\in D\cap C$, $h(x)=f(x)\lor g(x)$;
\item $\neg:A\to A$ is such that $\neg(C,f)=(C,\neg f)$;
\item If $(C,f)\mathbb{C}(D,g)$, then $(C,f)\lor(D,g)$ is defined and equals the $(Z,h)$ from item (v);
\item Any pair $(C,f),(D,g)$ such that $(C,f)\mathbb{C}(D,g)$ generates a Kleene sub-lattice of $\alga$, in the sense that polynomials in $(C,f),(D,g)$ are mutually commeasurable and, moreover, they form a Kleene lattice w.r.t. $\{\neg,\lor,0,1\}$.
\end{enumerate}
\item For any $i\in I$, $D_{i}=\{(C,f)\in A:i\in C\text{ and }f(i)=1\}$.
\end{enumerate}
\end{definition}
In particular, note that $\mathbb{O}=\neg\mathbb{I}$.

Let $\mathcal{A}=(I,\alga,\{D_{i}\}_{i\in I})$ be a paraconsistent partial referential matrix. For any $(C,f),(D,f)\in A$, set \[(C,f)\precsim(D,g)\text{ iff }C\cap D\ne\emptyset\text{ and }f(i)\leq g(i),\text{ for any }i\in C\cap D.\]

From now on, we will denote by $\mathbf{Fm}_{\mathcal{L}}$ the absolutely free algebra in the signature $\mathcal{L}=\{\lor,\neg,0,1\}$ generated by an infinite countable set of variables $Var$.
\begin{definition} Let $\alga$ be the algebraic reduct of a partial referential matrix. a partial mapping $h:Fm_{\mathcal{L}} \to A$ is said to be a \emph{partial homomorphism} if the following hold:
\begin{enumerate}
\item $Var\subseteq Dom(h)$;
\item $\varphi\in Dom(h)$ iff $\neg\varphi\in Dom(h)$;  $h(\neg\varphi)=\neg h(\varphi)$.
\item $\varphi\lor\psi\in Dom(h)$ iff $\varphi,\psi\in Dom(h)$ and $h(\varphi)\mathbb{C}h(\psi)$;  $h(\varphi\lor\psi)=h(\varphi)\lor h(\psi)$.
\item $h(1)=\mathbb{I}$.
\end{enumerate}
\end{definition}
For any $\mathcal{A}=(I,\alga,\{D_{i}\}_{i\in I})$, we denote by $\mathrm{Hom}^{P}(\mathbf{Fm}_{\mathcal{L}},\alga)$ the class of partial homomorphisms from $\mathbf{Fm}_{\mathcal{L}}$ to $\alga$. The proof of the next remark is customary, so it is left to the reader.  
\begin{remark}\label{rem:evalextenduniqueparhom}For any $\mathcal{A}=(I,\alga,\{D_{i}\}_{i\in I})$, any mapping $h:Var\to A$ can be extended to a \emph{unique} $h^{*}\in\mathrm{Hom}^{P}(\mathbf{Fm}_{\mathcal{L}},\alga)$.
 
\end{remark}

Let $\mathcal{A}=(I,\alga,\{D_{i}\}_{i\in I})$ be a paraconsistent partial referential matrix. Let us define a relation $\vdash_{\mathcal{A}}\subseteq\mathscr{P}(\mathrm{Fm}_{\mathcal{L}})\times \mathrm{Fm}_{\mathcal{L}}$ upon setting, for any $\Gamma\cup\{\varphi\}\subseteq Fm_{\mathcal{L}}$, $\Gamma\vdash_{\mathcal{A}}\varphi$ if and only if the following condition holds:
\begin{small}
  \[\forall h\in\mathrm{Hom}^{P}(\mathbf{Fm}_{\mathcal{L}},\alga)\ \forall i\in I,\text{ if }\Gamma\subseteq Dom(h)\text{ and }h(\Gamma)\subseteq D_{i},\text{ then }\varphi\in Dom(h)\text{ and }h(\varphi)\in D_{i}.\]
\end{small}
We observe that any paraconsistent partial referential matrix $\mathcal{A}=(I,\alga,\{D_{i}\}_{i\in I})$ induces a \emph{consequence relation} in the sense that, for any $\Gamma\cup\Delta\cup\{\varphi\}\subseteq\mathrm{Fm}_{\mathcal{L}}$:
\begin{itemize}
\item  $\Gamma\vdash_{\mathcal{A}}\gamma$, for any $\gamma\in\Gamma$ (reflexivity);
\item If $\Gamma\vdash_{\mathcal{A}}\varphi$ and $\Gamma\subseteq\Delta$, then $\Delta\vdash_{\mathcal{A}}\varphi$ (monotonicity);
\item If $\Gamma\vdash_{\mathcal{A}}\delta$, for any $\delta\in\Delta$, and $\Delta\vdash_{\mathcal{A}}\varphi$, then $\Gamma\vdash_{\mathcal{A}}\varphi$ (transitivity).
\end{itemize}
However, $\vdash_{\mathcal{A}}$ need not be preserved under \emph{substitutions} (i.e. \emph{structural}, cf. \cite[Definition 1.5]{Font16}). So, $\vdash_{\mathcal{A}}$ might not be a \emph{logic} in the proper sense, as shown by the next example.
\begin{example}Let us consider the paraconsistent partial referential matrix $$\mathcal{A}=(I=\{r,s\},\alga,\{D_{r},D_{s}\})$$ where $\alga=(A,\mathbb{C},\lor,\neg,\mathbb{O},\mathbb{I})$  is such that
\begin{enumerate}
\item $\mathbb{I}$ resp. $\mathbb{O}$ is the constant $1$ resp.  the constant $0$ function over $I$;
\item $A=\{\mathbb{I},\mathbb{O},(\{r\},f_{a}),(\{r\},f_{\neg a}),(\{s\},f_{b})\}$, where $f_{a}(r)=1$, $f_{\neg a}(r)=0$, $f_{b}(r)=\frac{1}{2}=f_{\neg b}(r)$.
\item $\mathbb{C}=\{((C,f),(D,g))\in A^{2}:C\cap D\ne\emptyset\}$; and, for any $((C,f),(D,g))\in \mathbb{C}$:
\begin{equation}
 (C,f)\lor(D,g)=
\begin{cases}
\mathbb{I} & \text{ if }(C,f)=\mathbb{I}\text{ or }(D,g)=\mathbb{I}\\
(C,f) & \text{ if }(D,g)=\mathbb{O};\\
(D,g) & \text{ if }(C,f)=\mathbb{O};\\
\mathbb{I} & \text{ if }(C,f)=(\{r\},f_{a}), (D,g)=(\{r\},f_{\neg a})\\
\mathbb{I} & \text{ if }(C,f)=(\{r\},f_{\neg a}), (D,g)=(\{r\},f_{a})\\
(C,f) & \text{ if }(C,f)=(D,g),\\
\text{undefined}, & \text{otherwise}
\end{cases}
\end{equation}
\item $D_{r}=\{\mathbb{I},(\{r\},f_{a})\}$ and $D_{s}=\{\mathbb{I}\}$.
\end{enumerate}
Note that, in particular, $\lor$ is not defined for the pair $(\{r\},f_{a}),(\{s\},f_{b})$.\\ 
Now, it can be seen that, for any $x,y\in Var$, it holds that 
\[x\vdash_{\mathcal{A}}(x\lor\neg x)\lor y.\]
For if $h\in\mathrm{Hom}^{P}(\mathbf{Fm}_{\mathcal{L}},\alga)$, then, for any $i\in I$, $h(x)\in D_{i}$ implies $h(x\lor\neg x)=\mathbb{I}$ and so $h(x\lor\neg x)\mathbb{C}h(y)$. Therefore,  $(x\lor\neg x)\lor y\in Dom(h)$ and $h((x\lor\neg x)\lor y)\in D_{i}$. Nevertheless, let us consider an arbitrary substitution $\sigma$ such that $\sigma: x\mapsto x$, $y\mapsto u\lor v$, where $u,v$ are fresh variables. Let $h$ be such that $h: x\mapsto (\{r\},f_{a}),u\mapsto (\{r\},f_{a}), v\mapsto (\{s\},f_{b})$. Then the partial homomorphism $h^{*}$ induced by $h$ satisfies $x\in Dom(h^{*})$ and $h^{*}(x)\in D_{r}$ but $(x\lor\neg x)\lor(u\lor v)\notin Dom(h^{*})$, since $u\lor v\notin Dom(h^{*})$. We conclude $x\nvdash_{\mathcal{A}}(x\lor\neg x)\lor(u\lor v)$. 
\end{example}
However, if one aims at obtaining a full-fledged logic, another route is possible. Let  $\mathcal{A}=(I,\alga,\{D_{i}\}_{i\in I})$ be a paraconsistent partial referential matrix.
Let us define a relation $\vdash_{\mathcal{A}}^{*}\subseteq\mathscr{P}(\mathrm{Fm}_{\mathcal{L}})\times \mathrm{Fm}_{\mathcal{L}}$ upon setting, for any $\Gamma\cup\{\varphi\}\subseteq \mathrm{Fm}_{\mathcal{L}}$,  
\[\Gamma\vdash_{\mathcal{A}}^{*}\varphi\text{ iff }\Gamma\vdash_{\mathcal{A}}\varphi\text{ and }V(\varphi)\subseteq V(\Gamma)\] where, for any $\Gamma\subseteq\mathrm{Fm}_{\mathcal{L}}$, $V(\Gamma)=\bigcup_{\gamma\in\Gamma}V(\gamma)$ and $V(\alpha)$ is the set of propositional variables occurring in $\alpha$, for any $\alpha\in\mathrm{Fm}_{\mathcal{L}}$.
\begin{proposition}For any paraconsistent partial referential matrix $\mathcal{A}=(I,\alga,\{D_{i}\}_{i\in I})$, $\langle \mathbf{Fm}_{\mathcal{L}},\vdash_{\mathcal{A}}^{*}\rangle$ is a sentential logic.
\end{proposition}
\begin{proof}
Showing that $\vdash_{\mathcal{A}}^{*}$ is a consequence relation is straightforward and so it is left to the reader. Let $\sigma$ be an arbitrary substitution and suppose that $\Gamma\vdash_{\mathcal{A}}^{*}\varphi$. Since $V(\varphi)\subseteq V(\Gamma)$, one has $V(\sigma(\varphi))\subseteq V(\sigma(\Gamma))$. Now, let $h\in \mathrm{Hom}^{P}(\mathbf{Fm}_{\mathcal{L}},\alga)$ and let $i\in I$. Suppose that $\sigma(\Gamma)\subseteq Dom(h)$ and $h(\sigma(\Gamma))\subseteq D_{i}$. Then, for any $x\in V(\Gamma)=V(\Gamma\cup\{\varphi\})$,  $\sigma(x)\in Dom(h)$. Now, let $g:Var\to A$ be the mapping such that $g(x)=h\sigma(x)$, if $x\in V(\Gamma)$, and $g(x)=h(x)$ otherwise. $g$ can be extended to a partial homomorphism $g^{*}$ such that $\Gamma\subseteq Dom(g^{*})$ and $g^{*}(\Gamma)=h\sigma(\Gamma)\subseteq D_{i}$. Therefore, we conclude that $\varphi\in Dom(g^{*})$ and $g^{*}(\varphi)\in D_{i}$. But then $\sigma(\varphi)\in Dom(h)$ and $h\sigma(\varphi)=g^{*}(\varphi)\in D_{i}$.
\end{proof}
Since investigating consequence relations introduced so far goes beyond the scope of the present paper, we postpone an inquiry into their properties to a future work.

\subsection{Unsharp orthogonal posets as paraconsistent partial referential matrices}

In what follows, we show that any unsharp orthogonal poset yields the algebraic reduct of a paraconsistent partial referential matrix. Consequently, any UOP determines an algebra of partial propositions which, in view of the above discussion, yields  a consequence relation with a quite intuitive flavor. 

Let $\alga\in\mathcal{KL}$. Recall that $F\subseteq A$ is said to be a lattice-filter provided that, for any $x,y\in A$, $x,y\in F$ iff $x\land y\in F$. For any $x\in A$, we set $\uparrow{x}:=\{y\in A:x\leq y\}$. Obviously, $\uparrow{x}$ is a lattice-filter of $\alga$. Furthermore, $F$ is said to be \emph{prime} provided that, for any $x,y\in A$, $x\lor y\in F$ iff $x\in F$ or $y\in F$. From now on, given $\alga\in\mathcal{KL}$. we will denote by $\mathrm{F}_{\alga}$ and $\mathrm{PF}_{\alga}$ the family of lattice-filters and prime lattice-filters over $\alga$, respectively. Whenever the reference to the underlying Kleene lattice $\alga$ will be clear, we will write simply $\mathrm{F}$ and $\mathrm{PF}$ in place of $\mathrm{F}_{\alga}$ resp. $\mathrm{PF}_{\alga}$. The following fact is well known (see e.g. \cite{Balbes}).
\begin{lemma}[Prime Filter Theorem]\label{thm:primfiltthm}Let $\alga\in\mathcal{KL}$. If $F\in\mathrm{F}_{\alga}$ is such that $x\notin F$, then there exists $G\in\mathrm{PF}_{\alga}$ such that $F\subseteq G$ and $x\notin P$.
\end{lemma}
In the sequel, we will freely make use of Lemma \ref{thm:primfiltthm} without further reference.

From now on, we consider an arbitrary $\alga\in\mathcal{UOP}$ and let $\mathfrak{S}(\alga)=\{\K_{i}\}_{i\in I}$. Let $\mathbb{F}_{i}:=\{F\subseteq K_{i}:  F\in\mathrm{PF}_{\K_{i}}\}$, for any $i\in I$. For any $a\in A$, let $$I_{a}=\{i\in I:a\in K_{i}\}.$$
Now let us consider, for any $a\in A$, $\mathbb{F}_{a}=\bigcup_{i\in I_{a}}\mathbb{F}^{}_{i}$. Moreover, for any $a\in A$, let $f_{a}:\mathbb{F}_{a}\to K_{3}$ be the mapping such that, for any $F\in \mathbb{F}_{a}$:
  \begin{equation}
    f_{a}(F) =
    \begin{cases}
      1 &  \text{if}\ a\in F\text{ and }a'\notin F \\
      0 & \text{if}\ a'\in F\text{ and } a\notin F \\
      \frac{1}{2} & \text{otherwise}
    \end{cases}
  \end{equation}
Now, let $\overline{A}=\{(\mathbb{F}_{a},f_{a}):a\in A\}$. For any $(\mathbb{F}_{a},f_{a}),(\mathbb{F}_{b},f_{b})\in\overline{A}$, we set $(\mathbb{F}_{a},f_{a})\mathbb{C}(\mathbb{F}_{b},f_{b})$ provided that $\mathbb{F}_{a}\cap \mathbb{F}_{b}\neq\emptyset$.

Furthermore, let us define the following operations: $\neg(\mathbb{F}_{a},f_{a}):=(\mathbb{F}_{a'},f_{a'}$), and, for any $(\mathbb{F}_{a},f_{a}),(\mathbb{F}_{b},f_{b})\in \overline{A}$ such that $(\mathbb{F}_{a},f_{a})\mathbb{C}(\mathbb{F}_{b},f_{b})$, $$(\mathbb{F}_{a},f_{a})\sqcup(\mathbb{F}_{b},f_{b})=(\mathbb{F}_{a\lor^{\K_{i}}b},f_{a\lor^{\K_{i}}b}),$$  where $\K_{i}$ is an arbitrary Kleene sub-lattice of $\alga$ such that $a,b\in K_{i}$. Note that, since any $\K_{i}\in\mathfrak{S}(\alga)$ is a Kleene sub-lattice, if $a,b\in\K_{i}\cap\K_{j}$, then $a\lor^{\K_{i}}b = a\lor^{\K_{j}}b$, for any $i,j\in I$. In view of this fact, whenever no danger of confusion will be impending, unnecessary superscripts will be omitted.  Also, set $$\mathbb{O}:=(\mathbb{F}_{0},f_{0})\text{ and }\mathbb{I}:=(\mathbb{F}_{1},f_{1}).$$
The next lemma ensures that $\sqcup$ and $\neg$ are well-defined.
\begin{lemma}\label{lem:embedd}For any $(\mathbb{F}_{a},f_{a}),(\mathbb{F}_{b},f_{b})\in \overline{A}$, $$(\mathbb{F}_{a},f_{a})=(\mathbb{F}_{b},f_{b})\text{ iff }a=b.$$
\end{lemma}
\begin{proof}Concerning the non trivial direction, suppose towards a contradiction that $a\ne b$ but $(\mathbb{F}_{a},f_{a})=(\mathbb{F}_{b},f_{b})$. If $I_{a}\cap I_{b}= \emptyset$, then $a,b\notin\{0,1\}$. Upon considering $F\in\mathrm{PF}_{\K_{i}}$, for some $\K_{i}$ such that $i\in I_{a}$ obtained e.g. by extending $\uparrow{a}$ to a prime lattice-filter, we have that $F\notin\mathbb{F}_{b}$, and so $\mathbb{F}_{a}\ne\mathbb{F}_{b}$, a contradiction. Therefore, let us suppose w.l.o.g. that $I_{a}\cap I_{b}\ne \emptyset$. Let $i\in I_{a}\cap I_{b}$. Since $a\ne b$, then one must have $a\nleq b$ or $b\nleq a$. Suppose w.l.o.g. that $a\nleq b$. Therefore, by Lemma \ref{thm:primfiltthm}, $\uparrow {a}\in\mathrm{F}_{\K_{i}}$ can be extended to a prime lattice-filter $F\in\mathbb{F}_{i}$ containing $a$ but not $b$. Now, by hypothesis $f_{a}(F)=f_{b}(F)$. Since $a\in F$, one must have $f_{a}(F)\in\{1,\frac{1}{2}\}$. If $f_{b}(F)=f_{a}(F)=1$, then $b\in F$, a contradiction. Therefore, $f_{a}(F)=f_{a}(F)=\frac{1}{2}$, $a,a'\in F$ and $b,b'\notin F$. But then we reach again a contradiction, since $a,a'\in F$ implies that $b\lor b'\geq a\land a'\in F$ and, since $F$ is prime, $b\in F$ or $b'$ in $F$.
\end{proof}
\begin{lemma}$\overline{\alga}=(\overline{A},\mathbb{C},\sqcup,\neg,\mathbb{O},\mathbb{I})$ is the algebraic reduct of a paraconsistent partial referential matrix.
\end{lemma}
\begin{proof}
We show that item 2 of Definition \ref{def: parparrefmatr} hold. (i)-(iv),(vi),(vii) are obvious. Concerning (v), suppose that $(\mathbb{F}_{a},f_{a})\mathbb{C}(\mathbb{F}_{b},f_{b})$, i.e. $I(a)\cap I(b)\ne\emptyset$. We show that $\sqcup$ satisfies the requirements of (v). Note that $\mathbb{F}_{a}\cap \mathbb{F}_{b}\subseteq \mathbb{F}_{a\lor b}$, for any $i\in I(a)\cap I(b)$. Next we show that, for any $F\in \mathbb{F}_{a}\cap \mathbb{F}_{b}$ such that $F\in\mathrm{PF}_{\K_{i}}$ ($i\in I(a)\cap I(b)$), $f_{a\lor b}(F)=f_{a}(F)\lor f_{b}(F)$. If $f_{a}(F)=1$ or $f_{b}(F)=1$,  then $a\in F$ and $a'\notin F$ or $b\in F$ and $b'\notin F$. This fact clearly implies $a\lor b\in F$ and $(a\lor  b)'=a'\land b'\notin F$. So, we can assume that neither $f_{a}(F)=1$ nor $f_{b}(F)=1$. If $f_{a}(F)=0=f_{b}(F)$, then $(a\lor^{\K_{i}}b)'=a'\land b'\in F$. Now, if $a\lor b\in F$, then $a\in F$ or $b\in F$, since $F$ is prime. A contradiction. Therefore, $a\lor b\notin F$, i.e. $f_{a\lor b}(F)=0$. Now, suppose that $f_{a}(F)\ne 0$ or $f_{b}(F)\ne 0$. So $f_{a}(F)=\frac{1}{2}$ or $f_{b}(F)=\frac{1}{2}$. Let us assume without loss of generality that $f_{a}(F)=\frac{1}{2}$. This means that 
(a) $a, a'\in F$; or (b) $a,a'\notin F$.\\
(a) One has that $a\leq a\lor b\in F$. Now, if $b'\notin F$, one has that $b\in F$, since $b\lor b'\geq a\land a'\in F$. So, $f_{b}(F)=1$, against our assumptions.
Therefore, $b'\in F$ and $(a\lor b)'\in F$.\\
(b) If $a,a'\notin F$, then clearly $(a\lor b)'\notin F$. If $a\lor b\in F$, then we must have $b\in F$. If $b'\in F$, then $a\lor a'\in F$ and so $a\in F$ or $a'\in F$, contradicting our hypothesis. So $b'\notin F$. But then $f_{b}(F)=1$, again a contradiction. We conclude $a\lor b,(a\lor b)'\notin F$.\\
To prove uniqueness, suppose that $(\mathbb{F}_{c},f_{c})$ is such that $\mathbb{F}_{a}\cap\mathbb{F}_{b}\subseteq \mathbb{F}_{c}$ and, for any $F\in \mathbb{F}_{a}\cap \mathbb{F}_{b}$, $f_{c}(F)=f_{a}(F)\lor f_{b}(F)$. We show that $(\mathbb{F}_{c},f_{c})=(\mathbb{F}_{a\lor b},f_{a\lor b})$, i.e. $a\lor b=c$, by Lemma \ref{lem:embedd}. Let $i\in I(a)\cap I(b)$ be such that $a,b,c\in K_{i}$. Suppose towards a contradiction that $a\lor b\ne c$. If $a\lor b\not\leq c$, consider the filter $\uparrow{a\lor b}\in\mathrm{F}_{\K_{i}}$. $\uparrow a\lor b$ can be extended to a prime lattice-filter $F$ such that $c\notin F$. Since $f_{c}(F)=f_{a}(F)\lor f_{b}(F)=f_{a\lor b}(F)$, then $f_{a\lor b}(F)\ne 1$, otherwise $c\in F$, a contradiction. So one must have $f_{a\lor b}(F)=\frac{1}{2}$, i.e. $a\lor b,(a\lor b)'\in F$. But then $c\lor c'\in F$ and $c'\in F$, since $F$ is prime. But this means that $f_{c}(F)=0$, which is impossible. We conclude that $a\lor b\leq c$. Similarly, we prove that $c\leq a\lor b$.\\
(viii) Just note that, for any $(\mathbb{F}_{a},f_{a}), (\mathbb{F}_{a},f_{a})\in \overline{A}$ such that $(\mathbb{F}_{a},f_{a})\mathbb{C}(\mathbb{F}_{a},f_{a})$, and any polynomial $p(x,y)$ in the language $\{\lor,',0,1\}$, ($\ast$) $p((\mathbb{F}_{a},f_{a}), (\mathbb{F}_{a},f_{a}))=(\mathbb{F}_{p(a,b)},f_{p(a,b)})$ and, of course, for any pair of polynomials $p_{1}(x,y),p_{2}(x,y)$, $I(p_{1}(a,b))\cap I(p_{2}(a,b))\ne\emptyset$ and so $\mathbb{F}_{p_{1}(a,b)} \cap \mathbb{F}_{p_{2}(a,b)}\ne\emptyset$. Moreover, by ($\ast$) and Lemma \ref{lem:embedd}, it is easily seen that axioms of Kleene lattices hold.\\
Now, upon considering, for any $F\in\mathbb{F}=\bigcup_{i\in I}\mathbb{F}_{i}$, $D_{F}:=\{(\mathbb{F}_{a},f_{a}):F\in\mathbb{F}_{a}\text{ and }f_{a}(F)=1\}$, one has that $(\mathbb{F},\overline{\alga},\{D_{F}\}_{F\in\mathbb{F}})$ is a paraconsistent partial referential matrix.
\end{proof}
Recall that, for any $(\mathbb{F}_{a},f_{a}),(\mathbb{F}_{b},f_{b})\in \overline{A}$, $(\mathbb{F}_{a},f_{a})\precsim(\mathbb{F}_{b},f_{b})$ is short for $(\mathbb{F}_{a},f_{a})\mathbb{C}(\mathbb{F}_{b},f_{b})$ and $f_{a}(F)\leq f_{b}(F)$, for any $F\in\mathbb{F}_{a}\cap\mathbb{F}_{b}$. It is easily seen that $\precsim$ is reflexive and antisymmetric.

\begin{lemma}\label{lem: 1}$(\mathbb{F}_{a},f_{a})\precsim (\mathbb{F}_{b},f_{b})$ implies $a\leq b$. Moreover, if $\alga$ is tame, then the converse holds as well.
\end{lemma}
\begin{proof}
Suppose that $(\mathbb{F}_{a},f_{a})\precsim (\mathbb{F}_{b},f_{b})$. This means that $\mathbb{F}_{a}\cap\mathbb{F}_{b}\ne\emptyset$ and $f_{a}(F)\leq f_{b}(F)$, for any $F\in\mathbb{F}_{a}\cap\mathbb{F}_{b}$. Let $i\in I_{a}\cap I_{b}$. If $a\not\leq b$, then $\uparrow a$ in $\K_i$ can be extended to some prime lattice-filter $F\in\mathbb{F}_{i}$ such that $b\notin F$. However, since $f_{a}(F)\leq f_{b}(F)$, and $f_{a}(F)\in\{1,\frac{1}{2}\}$, we must have that $f_{a}(F)=\frac{1}{2}$ with $a,a'\in F$. But then, by regularity, $b'\in F$, since $b\lor b'\in F$ and $b\notin F$. So $f_{b}(F)=0$, a contradiction. Concerning the moreover part, if $a\leq b$, then there exists $i\in I_{a}\cap I_{b}$ such that $a\leq^{\K_{i}}b$. So $\mathbb{F}_{a}\cap\mathbb{F}_{a}\ne\emptyset$ Now, suppose that there exists $F\in\mathbb{F}_{a}\cap\mathbb{F}_{b}$ such that $f_{a}(F)>f_{b}(F)$, where $F\in\mathrm{PF}_{\K_{i}}$, for some $i\in I(a)\cap I(b)$. Note that if $f_{a}(F)=1$, then $b\in F$ and so one must have $f_{b}(F)=\frac{1}{2}$. But then $b'\in F$ and so $a'\in F$, a contradiction. So one must have $f_{a}(F)=\frac{1}{2}$ and $f_{b}(F)=0$. Reasoning as above, also in this case we reach a contradiction.
\end{proof}
We conclude that any UOP $\alga$ yields the algebraic reduct of a paraconsistent partial referential matrix. As an immediate application of this fact, we have the following corollary.
\begin{corollary}Any type I or type II fuzzy quantum poset may be regarded as the algebraic reduct of a paraconsistent partial referential matrix.
\end{corollary}
Moreover, when dealing with tame UOPs, we have some more.  
\begin{lemma}\label{lem:precsimpartordtame} Let $\alga\in\mathcal{UOP}$. If $\alga$ is tame then $\precsim$ over $\overline{\alga}$ is transitive, i.e. it is a partial order.
\end{lemma}
\begin{proof}
It follows from Lemma \ref{lem: 1}. 
\end{proof}
\begin{lemma}\label{lem:proclo} Let $\alga$ be a tame unsharp orthogonal poset. Then the following hold:
\begin{enumerate}
\item $\mathcal{S}(\overline{\alga})=(\overline{A},\precsim,\neg,\mathbb{O},\mathbb{I})$ is an unsharp orthogonal poset;
\item $\mathcal{S}(\overline{\alga})\cong\alga$.
\end{enumerate}
\end{lemma}
\begin{proof}
(1)Observe that, by construction, and since $\precsim$ is a partial order, $\mathcal{S}(\overline{\alga})=(\overline{A},\precsim,\neg,\mathbb{O},\mathbb{I})$ is a bounded poset with antitone involution. Now, suppose that $(\mathbb{F}_{a},f_{a})\precsim\neg(\mathbb{F}_{b},f_{b})=(\mathbb{F}_{b'},f_{b'})$. Then, since $\mathbb{F}_{b}=\mathbb{F}_{b'}$, we have $\mathbb{F}_{a}\cap\mathbb{F}_{b}\ne\emptyset$. So $(\mathbb{F}_{a},f_{a})\mathbb{C}(\mathbb{F}_{b},f_{b})$, $(\mathbb{F}_{a},f_{a})\sqcup(\mathbb{F}_{b},f_{b})$ is defined and $(\mathbb{F}_{a},f_{a})\sqcup(\mathbb{F}_{b},f_{b})=(\mathbb{F}_{a\lor b},f_{a\lor b})$. Clearly, $(\mathbb{F}_{a},f_{a}),(\mathbb{F}_{b},f_{b})\precsim(\mathbb{F}_{a\lor b},f_{a\lor b})$. Moreover, if $(\mathbb{F}_{a},f_{a}),(\mathbb{F}_{b},f_{b})\precsim(\mathbb{F}_{c},f_{c})$, then $a,b\leq c$ and $a\lor b\leq c$ (Lemma \ref{lem: 1}). Therefore, $(\mathbb{F}_{a\lor b},f_{a\lor b})\precsim(\mathbb{F}_{c},f_{c})$ (again by Lemma \ref{lem: 1}). Finally, let us show that Kleene condition holds. $\neg((\mathbb{F}_{a},f_{a})\sqcup \neg(\mathbb{F}_{a},f_{a}))=(\mathbb{F}_{a\land a'},f_{a\land a'})$. Upon noticing that $a\land a'\leq b\lor b'$, the desired conclusion follows by Lemma \ref{lem: 1}.\\
Concerning (2), it is easily seen that the map $a\mapsto(\mathbb{F}_{a},f_{a})$ is an ortho-isomorphism, by Lemma \ref{lem:embedd} and Lemma \ref{lem: 1}. 
\end{proof}
In view of results from Section \ref{sp-orthomodular}, since any sp-orthomodular lattice is a tame UOP, then 
\begin{theorem}Let $\alga\in\mathcal{PKL}$. If $\alga\in\mathcal{SPO}$, then $\mathcal{S}(\overline{\alga})$ is an UOP and $\alga\cong\mathcal{S}(\overline{\alga})$.
\end{theorem}
In other words, any sp-orthomodular lattice is orthoisomorphic to an unsharp orthogonal poset of partial propositions with ``truth values'' ranging in $\{0,\frac{1}{2},1\}$.

\section{Conclusion and future research}\label{sec:conclusion}
This work has been devoted to investigate sub-classes of unsharp orthogonal posets which are pastings of their maximal Kleene sub-lattices. Specifically, we introduced sp-orthomodular lattices as the class of pseudo-Kleene lattices whose order induces, and it is induced by, the order of its Kleene sub-lattices. It turned out that sp-orthomodular lattices can be considered as genuine quantum structures, since they may be regarded as abstractions of ``concrete'' structures of  experimental propositions concerning unsharp measurements over a physical (quantum) system. Also, we have investigated some of their prominent properties. Interestingly enough, the investigation of sp-orthomodular lattices unveils order-theoretical and algebraic properties of ``pastings'' of Kleene lattices that, when dealing with ortholattices, remain somehow ``hidden''. See, for example, forbidden configurations and properties of localizers.

As it often happens, this work provides more open problems than answers. Besides unanswered questions raised along the paper, we list hereafter some issues one could naturally address.
\begin{itemize}
\item[-] Provide a general theory of commutativity for sp-orthomodular lattices. It is interesting to observe that, when dealing with the general framework we are concerned with in this paper, we have that the intersection of \emph{all} Kleene blocks of an sp-orthomodular lattice $\alga$ need not be a Boolean algebra, as it happens for OMLs.
\item[-] Provide a logic which is complete w.r.t. a suitable matrix semantics having sp-orthomodular lattices as algebraic reducts. For example, providing a non-distributive variant of strong Kleene logic as well as of Priest's Logic of Paradox (see e.g. \cite{Albuquerque}) might be tasks deserving some attention.
\item[-] Introduce a suitable notion of an sp-orthomodular \emph{poset} playing for sp-orthomodular lattices the same role played by OMPs w.r.t. OMLs. This can be accomplished by considering the class of unsharp orthogonal posets satisfying the following condition
\[\text{If }x\leq y,\ \text{then }y\land(x\lor x')\text{ exists and }y\land(x\lor x')=x\lor(x'\land y).\label{spom}\tag{spOM}\] 
It is easily seen that type I and type II fuzzy quantum posets, once regarded as unsharp orthogonal posets, satisfy \eqref{spom}, namely they are examples of sp-orthomodular posets. Therefore, an inquiry into properties of these structures seems to be a worth task which may yield very general results.
\item[-] Investigate the structure theory for sp-orthomodular lattices.
\end{itemize}
\subsection*{Acknowledgements}This work has been funded by the European Union - NextGenerationEU under the Italian Ministry of University and Research (MUR) National
Innovation Ecosystem grant ECS00000041 - VITALITY - CUP C43C22000380007. The research of D.~Fazio has been partly  carried out during a research stay at the Nicolaus Copernicus University of Toru\'n (PL) under the Excellence-Initiative Research Programme. The author thanks R.~Gruszczy\'nski, T. Jarmu\.zek, M. Klonowski, and all the members of the Department of Logic for the valuable support received during the research stay. Finally, the authors gratefully thanks R.~Giuntini, A.~Ledda, F.~Paoli, G. St. John and G.~Vergottini for the insightful discussions and precious suggestions concerning the topics of the present work.


\begin{thebibliography}{100}

\bibitem{Albuquerque}
H. Albuquerque, A. P\v{r}enosil, U. Rivieccio, An Algebraic View of Super-Belnap Logics, \emph{Studia Logica}, 105(6), 2017, pp. 1051-1086. 

\bibitem{Balbes}
R. Balbes, P. Dwinger, \emph{Distributive lattices}, University of Missouri Press, 1975.\bibitem{Beran}
 L.~Beran, \emph{Orthomodular Lattices: algebraic approach}, D. Reidel Publishing
 Company, 1984.
 
\bibitem{Bernardinello}
L. Bernardinello, L. Pomello, S. Rombol\`a, On Orthomodular Posets Generated by Transition Systems, \emph{Electronic Notes in Theoretical Computer Science}, 270(1), 2011, pp. 147-154.
 
\bibitem{BV36}Birkhoff G., von Neumann J., ``The logic of quantum mechanics'', \emph{Annals of Mathematics}, 37, 1936, pp. 823-843.

\bibitem{Boicescu}
V. Boicescu, A. Filipoiu, G. Georgescu, S. Rudeanu, \emph{{\L}ukasiewicz-Moisil Algebras}, North-Holland, Amsterdam, 1991.

\bibitem{BH00}
Bruns G., Harding J., Algebraic Aspects of
Orthomodular Lattices, In: Coecke B., Moore D., Wilce A.
(eds) \emph{Current Research in Operational Quantum Logic}. Fundamental Theories of Physics, vol 111. Springer, Dordrecht, 2000.

\bibitem{BS}S. Burris, H. P. Sankappanavar, \emph{A Course in Universal Algebra}, {Springer}, Berlin, 1981.

\bibitem{Bush1}
P. Busch, "Unsharp Reality and the Question of Quantum Systems," in \emph{Symposium on the Foundations of Modern Physics}, June 8, 1987, Joensuu, Finland.

\bibitem{Bush2}
P. Busch, G. Jaeger, Unsharp Quantum Reality. \emph{Found Phys}, 40, 2010, pp. 1341--1367.

\bibitem{Chajdaorthogonal}
I. Chajda, An algebraic axiomatization of orthogonal posets, \emph{Soft Comput} 18, 2014, pp. 1--4. 

\bibitem{Chajda2}
I. Chajda, A note on pseudo-Kleene algebras, \emph{Acta Univ. Palacki. Olomuc, Fac. Rerum Nat., Math.}, 55, 2016, pp. 39--45.

\bibitem{Chajda1}
I. Chajda, H. L\"{a}nger. Orthomodular lattices can be converted into left residuated l-groupoids. \emph{Miskolc Mathematical Notes}, 18(2), 2017, pp. 685-689.

\bibitem{FaCha}
I. Chajda, D. Fazio, On residuation in paraorthomodular lattices, \emph{Soft Computing}, 24, 2020, pp. 10295-10304.
 
\bibitem{Chajda}
I.~Chajda, D.~Fazio, H.~L\"anger, A.~Ledda, J.~Paseka, Algebraic properties of paraorthomodular posets, \emph{Logic Journal of the IGPL}, Vol. 30(5), 2022, pp. 840–869. 

\bibitem{ChFazLeLaPa}
I.~Chajda, D.~Fazio, H.~L\"anger, A.~Ledda, J.~Paseka, Implication in sharply paraorthomodular and relatively paraorthomodular posets, arXiv:2301.09529.
\bibitem{Cignoli1}
R. Cignoli, Boolean elements in Luckasiewicz algebras i. Proc. Japan Acad., 41, 1965, pp. 670--675.

\bibitem{Cignoli2}
R. Cignoli, The class of Kleene algebras satisfying an interpolation property and Nelson algebras. \emph{Algebra Universalis}, 23(3), 1986, pp. 262--292.

\bibitem{Cignoli}
R. Cignoli, The algebras of {\L}ukasiewicz many-valued logic: A historical overview, in S. Aguzzoli, A. Ciabattoni, B. Gerla, C. Manara and V. Marra (eds.)\emph{Algebraic and Proof-Theoretic Aspects of Non-classical  logics}, LNAI 4460, SPringer, Berlin, 2007, pp. 69-83.


\bibitem{Czela1981b}
 J.~Czelakowski, Partial Referential Matrices, in E.~Beltrametti, Bas C.~Van Fraassen (eds.), \emph{Current issues in quantum logic},  Ettore Majorana International Science Series, 1981, pp. 131--146.
 
\bibitem{DallaChiara}
M. Dalla Chiara , R. Giuntini , R. Greechie, \emph{Reasoning in quantum theory: Sharp and unsharp quantum logics}, Springer Dordrecht, 2004.
  
\bibitem{DeGroote}
H. F. de Groote, On a canonical lattice structure on the effect algebra of a von Neumann algebra, arXiv:math-ph/0410018, 2005.

\bibitem{DvurChov}
A. Dvure\v{c}enskij, F. Chovanec, Fuzzy quantum spaces and compatibility. \emph{International Journal of Theoretical Physic}s, 27(9), 1988, pp. 1069-1082.  

\bibitem {DP00}Dvure\v{c}enskij A., Pulmannov\'{a} S., \emph{New Trends in Quantum Structures}, Kluwer, Dordrecht, 2000.

\bibitem{FaLePa2}
D. Fazio, A. Ledda, F. Paoli, On Finch's conditions for the completion of Orthomodular Posets, \emph{Foundations of Science}, 28, 2023, pp. 419-440.

\bibitem{FaLePa1}
D. Fazio, A. Ledda, F. Paoli, Residuated Structures and Orthomodular Lattices. \emph{Studia Logica}, 109, 2021, pp. 1201–1239.

\bibitem{Finch}
P. D. Finch, Quantum logic as an implication algebra. \emph{Bull. Austral. Math. Soc.}, 2, 1970, pp. 101-106.

\bibitem{Font16}
 J.~M.~Font, \emph{Abstract algebraic logic: an introductory textbook}. College Publications, 2016.

\bibitem{FB}D. J. Foulis, M. K. Bennett, Effect algebras and unsharp quantum logic, \emph{Foundations of Physics}, 24, 1994, pp. 1331-1352.

\bibitem{StJohn}
W. Fussner, G. St. John, Negative translations of orthomodular lattices and their logic. \emph{Electronic Proceedings in Theoretical Computer Science}, 343, 2021, pp. 37-49.

\bibitem{Genca}
J. Gen\v{c}a, Blocks of homogeneous effect algebras, \emph{Bull. Austral. Math. Soc.}, vol. 64 (2001), pp. 81--98.

\bibitem{GG05}R. Giuntini, H. Greuling, ``Toward a formal language for unsharp properties'', \emph{Foundations of Physics}, 19, 1989, pp. 931-945.

\bibitem{GiuLePa}
R. Giuntini, A. Ledda, F. Paoli, A New View of Effects
in a Hilbert Space, \emph{Studia Logica}, 104, 2016, pp. 1145-1177.

\bibitem{GiuMurPa}
R. Giuntini, C. Mure\c{s}an, F. Paoli, On PBZ*-Lattices, in M. Mojtahedi, S. Rahman, M. S. Zarepour (eds.), \emph{Mathematics, Logic, and their Philosophies. Essays in Honour of Mohammad Ardeshir}, Springer Cham, 2021.
 
\bibitem{Gratzer}
G. Gr\"atzer, \emph{Lattice Theory. Foundation}, Birkh\"auser, 2011.

\bibitem{Husimi}
K. Husimi, Studies on the foundations of quantum mechanics, I, Proc, \emph{Physics-Math. Soc. Japan}, 19, 1937, pp. 766--789.

\bibitem{Jonsson}
B.~J\'onsson, Distributive Sublattices of a Modular Lattice, \emph{Proceedings of the American Mathematical Society}, 6(5),  1955, pp. 682-688.

\bibitem{Kadison}
R. V. Kadison, Order properties of bounded self-adjoint operators, \emph{Proceedings of the American Mathematical Society} 2(3), 1951, pp. 505-510.

\bibitem{Kalman}
J. Kalman, Lattices with involution, \emph{Transactions of American Mathematical Society}, 87, 1958, pp.485--491.

\bibitem{Ka83}
Kalmbach G., \emph{Orthomodular Lattices}, London Math. Soc. Monographs, vol. VIII, Academic Press, London, 1983.

\bibitem{LeBaLong}
L. Ba Long, Fuzzy Quantum Posets, \emph{International Journal of Uncertainty, Fuzziness and Knowledge-Based Systems}, 7(4), 1999, pp. 317-325. 

\bibitem{LeBaLong1}
L. Ba Long, Compatibility and representation of observables on type I, II, III fuzzy quantum posets, \emph{Fuzzy Sets and Systems}, 61(2), 1994, pp. 215--224.

\bibitem{Ludwig}
G. Ludwig, \emph{Foundations of Quantum Mechanics}, vol. I and II, Springer, Berlin, 1983.

\bibitem{MK} 
G. W. MacKey, \emph{Mathematical Foundations of Quantum Mechanics}, Dover Publications, 2004.

\bibitem{Mittelstaedt}
 P.~Mittelstaedt, Classification of different areas of work afferent to quantum logic, in E.~Beltrametti, Bas C.~Van Fraassen (eds.), \emph{Current issues in quantum logic},  Ettore Majorana International Science Series, 1981, pp. 3--16.

\bibitem{Navara1}
M. Navara, Orthoalgebras as Pastings of Boolean Algebras, \emph{Int. J. Theor. Phys.}, 56, 2017, pp. 4126--4132.

\bibitem{Olson}
M. P. Olson, The selfadjoint operators of a von Neumann algebra form a conditionally complete lattice, \emph{Proc. Am. Math. Soc.}, 28(2), 1971, pp. 537--544.

\bibitem{Pulman}
S. Pulmannov\'a, Compatibility and partial compatibility
in quantum logics, \emph{Annales de l’I. H. P.}, section A, 34(4), 1981, p. 391-403.

\bibitem{Riecanova1}
Z. Rie\v{c}anov\'a, Generalization of Blocks for D-Lattices and Lattice-Ordered Effect Algebras, \emph{International Journal of Theoretical Physics}, 39, 2000, pp. 231--237.

\bibitem{Suppes}
P. Suppes, Logics Appropriate To Empirical Theories, in  J. W. Addison, L. Henkin, A. Tarski (eds.),
\emph{Studies in Logic and the Foundations of Mathematics,
The Theory of Models}, North-Holland, 2014, pp. 364-375.

\bibitem{Str01}
D. W. Stroock, \emph{A Concise Introduction to the Theory of Integration}, 3rd ed., Birkh\"auser, Basel, 1998.

\end{thebibliography}
\end{document}